\documentclass[12pt,reqno]{amsart}

\usepackage{amsfonts}
\usepackage{amscd}
\usepackage{amssymb}
\usepackage{amsfonts}
\usepackage{amscd}
\usepackage{amsmath} 
\usepackage{mathrsfs}
\usepackage{dsfont}
\usepackage{enumerate}
\usepackage{xcolor}
\usepackage{float} 
\usepackage[pdftex]{graphicx}
\allowdisplaybreaks
\usepackage{hyperref}

\date{\today}

\newtheorem*{theorem*}{Theorem}
\newtheorem{thm}{Theorem}[section]
\newtheorem{cor}[thm]{\bf{Corollary}}
\newtheorem{lem}[thm]{Lemma}

\newtheorem{prop}[thm]{Proposition}
\theoremstyle{definition}
\newtheorem{defn}[thm]{Definitions}
\theoremstyle{remark}
\newtheorem{rem}[thm]{\bf{Remark}}

\newtheorem{exmp}{\bf{Example}}[section]
\numberwithin{equation}{section}

\newcommand{\beas}{\begin{eqnarray*}}
	\newcommand{\eeas}{\end{eqnarray*}}
\newcommand{\bes} {\begin{equation*}}
	\newcommand{\ees} {\end{equation*}}
\newcommand{\be} {\begin{equation}}
	\newcommand{\ee} {\end{equation}}
\newcommand{\bea} {\begin{eqnarray}}
	\newcommand{\eea} {\end{eqnarray}}
\newcommand{\ra} {\rightarrow}

\newcommand{\R}{\mathbb R}

\newcommand{\Z}{\mathbb Z}

\newcommand{\la}{\lambda}
\newcommand{\C}{{\mathbb C}}

\newcommand{\N}{{\mathbb N}}

\newcommand{\Chi}{\mbox{\large$\chi$} }

\newcommand{\vp}{\varphi}

\renewcommand{\Re}{\operatorname{Re}}
\renewcommand{\Im}{\operatorname{Im}}
 
\usepackage{color}
\newcommand{\blue}[1]{\textcolor{blue}{#1}}

\title[Weighted estimates on Harmonic $NA$ group]
{Weighted estimates for Hardy-Littlewood maximal functions on Harmonic $NA $ groups }

\author[Ganguly, Rana, and Sarkar]{ Pritam Ganguly, Tapendu Rana, and Jayanta Sarkar}

\address{Pritam Ganguly  \endgraf Institut f\"ur Mathematik,	\endgraf Universit\"at Paderborn, 	\endgraf 33098 Paderborn, Germany.} \email{pritam1995.pg@gmail.com,  pritamg@math.upb.de}

\address{Tapendu Rana  \endgraf Department of Mathematics: Analysis, Logic and Discrete Mathematics,	\endgraf Ghent University, 	\endgraf Krijgslaan 281, Building S8, B 9000 Ghent, Belgium.} \email{tapendurana@gmail.com, tapendu.rana@ugent.be}

\address{Jayanta Sarkar \endgraf Department of Mathematics and Statistics, \endgraf Indian Institute of Science Education and Research Kolkata, \endgraf Mohanpur-741246, Nadia, West Bengal, India.} 
\email{jayantasarkarmath@gmail.com, jayantasarkar@iiserkol.ac.in}

\date{}

\begin{document}
	\maketitle

	\begin{abstract} 
   Our aim in this article is to study the weighted boundedness of the centered Hardy-Littlewood maximal operator in Harmonic $NA$ groups. Following Ombrosi et al. \cite{ORR}, we define a suitable notion of $A_p$ weights, and for such weights, we prove the weighted $L^p$-boundedness of the maximal operator. Furthermore, as an endpoint case, we prove a variant of the Fefferman-Stein inequality, from which vector-valued maximal inequality has been established. We also provide various examples of weights to substantiate many aspects of our results. In particular, we have shown certain spherical functions of the Harmonic $NA$ group constitute examples of $A_p$ weights. The purely exponential volume growth property of the Harmonic $NA$ group has played a crucial role in our proofs.     
	
 \end{abstract}
	
	\section{Introduction and main results}

The study of weighted inequalities in harmonic analysis can be traced back to the influential works by Fefferman-Stein \cite{FS}, and Muckenhoupt \cite{Muc} in 1971 and 1972, respectively. In these two seminal works, they provided a characterization of weights that allow the boundedness of the Hardy-Littlewood maximal operator on weighted $L^p$ spaces on $\R^d$. This characterization was achieved by introducing the so-called $ A_p$ class of weights. The significance of the $A_p$ class was soon further enhanced by the observation that many other operators exhibited similar weighted estimates. Consequently, the study of weighted estimates in $\R^d$, as well as in several other spaces, emerged as a vibrant and thriving research area in modern harmonic analysis, leading to numerous activities in the spaces of homogeneous and non-homogeneous types. This article is devoted to advancing this line of research in the context of certain Riemannian manifolds with non-positive sectional curvatures. We begin by providing some background that will make the discussion mathematically more rigorous.

  Let $(X, d, \mu)$ be a metric measure space such that $0<\mu(B(x,r))<\infty$, for any ball $B(x,r)$ centered at $x\in X$ and of radius $r>0$.  We denote by $L^p(\mu)$, the usual Lebesgue spaces with respect to the measure $\mu$ for $1\leq p\leq \infty.$ Throughout this article, by a weight on $X$ we mean a positive, locally integrable function $w$ on $X$, and for a measurable set $E\subseteq X$, the integral $\int_Ew(x)\:d\mu(x)$ is denoted by $w(E)$. For a locally integrable function $f$ on $X$, the centered Hardy-Littlewood maximal function $M_Xf$ is defined as 
	\begin{equation}\label{hlmax}
	   M_{X}f(x)=\sup_{r>0}\frac{1}{\mu(B(x,r))}\int_{B(x,r)}|f(y)|\;d\mu(y).
    \end{equation} 
    When $X=\R^d$ equipped with the usual Euclidean metric and the Lebesgue measure, $M_{\R^d}$ was introduced by Hardy and Littlewood in the case of $d=1$, and the higher dimensional case was studied by Wiener. 
It is well known that $M_{\R^d}$ is bounded on $L^p(\R^d)$, $1<p\leq\infty$, and of weak type $(1,1)$. In their remarkable work  \cite{FS}, extending the weak $(1,1)$ estimate by introducing weights, Fefferman and Stein proved the following inequality:
given a non-negative, locally integrable function $w$ on $\R^d$ the one has
\be\label{feffer}
w(\{x\in \R^d: M_{\R^d}f(x)>\lambda\})\lesssim_d \frac{1}{\lambda}\int_{\R^d}|f(x)|M_{\R^d}w(x)\:dx.
\ee
 It immediately follows from (\ref{feffer}) that $M_{\R^d}$ is a bounded operator from $L^1(\R^d, wdx)$ to $L^{1,\infty}(\R^d,wdx)$ whenever $M_{\R^d} w\leq w$ a.e., which is written symbolically as $w\in A_1$.
Fefferman and Stein \cite{FS} noted that the condition $w\in A_1$ is also necessary. 

For $1<p<\infty$, Muckenhoupt \cite{Muc} introduced the $A_p$ class of weights by the prescription that $w\in A_p$ if 
\begin{equation}\label{euclap}
\sup_{B}\left(\frac{1}{|B|}\int_{B}w(x)\:dx\right)\left(\frac{1}{|B|}\int_{B}w(x)^{-\frac{1}{p-1}}\:dx\right)^{p-1}<\infty,
\end{equation}
where the supremum is taken over all balls $B$ in $\R^d$. He proved that $w\in A_p$ is both necessary and sufficient for $M_{\R^d}$ to be bounded on $L^p(\R^d,wdx)$.
In proving both weighted and unweighted boundedness results for various operators, the Calderon-Zygmund decomposition and certain covering lemmas have been instrumental.
Since these two monumental works, the weighted inequalities, particularly the Fefferman-Stein inequality and related variants, have been investigated in many situations and for a number of operators; see, for instance, \cite{Bel, LP, OC} and references therein. The Calderon-Zygmund operators constitute an important class of such examples; see \cite{FS, LOR, LKP} and references therein.

Beyond Euclidean spaces, weighted norm inequalities for the Hardy-Littlewood maximal function have also been studied in the context of doubling metric measure spaces \cite[Chapter 2]{Hei}. We would also like to mention that a lot of work has been done in non-doubling metric measure spaces for suitable modification of the Hardy-Littlewood maximal function \cite{Hyt, NTV, Swa, Ste}. In these works, $\mu(B(x,r))$ has been replaced by $\mu(B(x,kr))$, for some $k\geq 2$, in the denominator of the right-hand side of (\ref{hlmax}).

Based on what has been discussed so far, the weak $(1,1)$ boundedness of the Hardy-Littlewood maximal operator might appear to have a connection with some doubling condition in the metric measure space under consideration.  In \cite[Theorem 1.5]{NT}, among other things, Naor and Tao have shown that there are metric measure spaces without any kind of doubling condition where the weak $(1,1)$ boundedness of the Hardy-Littlewood maximal operator holds true. More precisely, they developed an ingenious strategy based on combinatorial arguments to obtain the weak-type $(1,1)$ boundedness of the Hardy-Littlewood maximal operator on infinite rooted $k$-ary tree $T_k$, where $k\geq 2$ is an integer (a.k.a. the homogeneous tree of degree $k$), with bound independent of $k$. It is known that if $B(x,r)$ is a ball with center at $x\in T_k$ and radius $r\in \N$, then $|B(x,r)|\simeq k^r$. Due to this exponential growth, covering arguments do not work since the doubling condition or, even more generally, the upper doubling property introduced by Hyt\"{o}nen \cite{Hyt} completely fails. We remark that the weak $(1,1)$ boundedness of the maximal operator on $T_k$ can also be deduced from the work of Rochberg and Taibleson \cite{RT}. Later, a simplified proof was given by Cowling et al.\cite{CMS}. Their argument is based on the proof of weak $(1,1)$ boundedness of the Hardy-Littlewood maximal operator on the Harmonic $NA$ groups due to Anker, Damek, and Yacoub \cite{ADY}. The proof given by Anker et al. was inspired by Str\"{o}mberg's work on Riemannian symmetric spaces of non-compact type \cite{Str}. We shall discuss more on this in the later part of this section. Nevo and Stein \cite{NS} established the strong type $(p,p)$ estimate, where $p>1$, of the maximal function.

Motivated by the work of Naor and Tao, in \cite{ORS}, Ombrosi, Rivera-Ríos, and Safe initiated the study of weighted estimates for the Hardy-Littlewood maximal function on infinite rooted $k$-ary trees. Among other things, they proved the following strong version of the Fefferman-Stein inequality on $T_k$: given $s>1$, and a weight $w$ on $T_k$, the following holds:
\begin{equation}\label{fstree}
 w(\{x\in T_k: M_{T_k}f(x)>\lambda\})\leq \frac{C_{s}}{\lambda}\sum_{x\in T_k}|f(x)|M_{T_k}(w^s)(x)^{1/s}\:dx,   
\end{equation} 
where the constant $C_s\to\infty$ as $s\to 1$. Moreover, they showed that in general (\ref{fstree}) does not hold for $s=1$.
By appropriately extending the techniques developed by Naor and Tao,  Ombrosi and Rivera-Ríos introduced the following condition 
  \begin{equation}
 \mathds{1}\otimes w(\{(x,y)\in E\times F: d_{T_k}(x,y)=r\}) \lesssim k^{ r\beta} w(E)^{\frac{\alpha}{p}}w(F)^{1-\frac{\alpha}{p}},
\end{equation}
on a weight $w$, which they proved to be sufficient for the boundedness of $M_{T_k}$ on $L^p(w)$, for $p\in(1,\infty)$ and certain ranges of the parameters $\alpha$ and $\beta$; see \cite[Theorem 1.1]{ORR} for more details. In fact, they worked with the spherical maximal operator as the Hardy-Littlewood maximal operator is equivalent to this in the context of $T_k$.

The present article seeks to continue this trend by extending the aforementioned weighted estimates for Hardy-Littlewood maximal functions in the context of Harmonic $NA$ groups. We begin by introducing a few notations necessary to illustrate the body of research in this area and present the main results of this article.

Harmonic $NA$ groups also known as Damek-Ricci spaces, were introduced by Damek and Ricci in 1992 as a family of counterexamples to the Lichnerowicz conjecture in the non-compact case \cite{DR1}. A Harmonic $NA$ group $S$ is a semidirect product $N\rtimes A$, where $N$ is a $H$-type group, $A=(0,\infty)$, and the action of $A$ on $N$ is anisotropic dilations. For any unexplained notions and terminologies, we refer the reader to Section \ref{prelim}. It is known that $S$ is a non-unimodular solvable Lie group which is also a non-flat harmonic manifold with respect to a metric that is left invariant under the action of $S$. Despite being the most distinct prototypes, the rank one Riemannian symmetric spaces of non-compact type, which are contained within them as Iwasawa $NA$ groups, only make up a very small subclass \cite{ADY}.

It is known that the volume of the ball $|B(x,r)|\simeq e^{2\rho r}$, for $r\geq1$, where the positive number $2\rho$ is the so-called homogeneous dimension of $N$, which prevents $S$ from being doubling. This leads to the absence of an analogue of the Calderon-Zygmund decomposition or any useful covering lemma. Nonetheless, using a convolution inequality for semi-simple Lie groups, Clerc and Stein \cite{CS} obtained the $L^p$-boundedness for $p>1$ of the centered Hardy-Littlewood maximal operator on Riemannian symmetric spaces of noncompact type. By getting beyond this aforementioned difficulty, the endpoint case was subsequently solved by Str\"{o}mberg \cite{Str}. The weak $(1,1)$ estimate resulting from Str\"{o}mberg's work demonstrates that the centered Hardy-Littlewood maximal operator in this setting enjoys the same properties as in the case of the Euclidean space despite the stark contrasts between both of these spaces. We remark in passing by that Ionescu established $L^p$-boundedness of the non-centered Hardy-Littlewood maximal operator for $p>2$ on Riemannian symmetric spaces of noncompact type (the exponent is sharp when the rank is one) in his remarkable works \cite{Ior1, Iorh}. Adapting the argument of Str\"{o}mberg for the particular case of rank one, Anker, Damek, and Yacoub \cite{ADY}  proved the weak $(1,1)$ estimate for the centered Hardy-Littlewood maximal operator on Harmonic $NA$ groups.

Despite the above list of outstanding works, the study of weighted estimates of the Hardy-Littlewood maximal operators on symmetric spaces or on Harmonic $NA$ groups is an unexplored territory. Our aim in this paper is to investigate the weighted boundedness of the centered Hardy-Littlewood maximal operator on Harmonic $NA$ groups. Motivated by the works of Ombrosi, Rivera-Ríos, and Safe \cite{ORR, ORS}, we shall define a notion of $A_p$ weights, $1<p<\infty$, on Harmonic $NA$ groups and prove weak and strong type $(p,p)$ estimates for the maximal operator and a variant of Fefferman-Stein inequality. We will also provide examples and counter-examples of various weights in order to illustrate various aspects of our results and demonstrate the difficulty that is stopping us from extending these results to symmetric spaces of higher rank.   
 
We now briefly, but in precise terms, discuss the main results of this paper. In what follows, $S=NA$ is a Harmonic $NA$ group equipped with a left-invariant Riemannian metric $d$, and a left Haar measure. From now onwards, we denote the centered Hardy-Littlewood maximal operator on $S$ by $M$. In order to study the behavior of $M$, following \cite{ADY}, it is customary to consider separately the local and large-scale variant of $M$, denoted by $M^0$ and $M^{\infty}$ respectively, and defined by 
\begin{align*}
 &M^{0}f(x)=\sup_{0<r\leq T}\frac{1}{|B(x,r)|}\int_{B(x,r)}|f(y)|\;dy;\\&M^{\infty}f(x)=\sup_{r>T}\frac{1}{|B(x,r)|}\int_{B(x,r)}|f(y)|\;dy, 
 \end{align*}
 where $|A|$ stands for the left Haar measure of a Borel measurable set $A\subseteq S$ and $T>0$ is a fixed number. Since the left Haar measure on $S$ is \textit{locally doubling}, in view of techniques for doubling metric measure spaces, it is not hard to see that $M^0$ is bounded on $L^p(w)$, $1<p<\infty$, whenever $w$ satisfies the Euclidean type $A_p$ condition locally, i.e., 
 \begin{align}\label{aploc}
     \sup_{0<r(B)\leq T}\left(\frac{1}{|B|}\int_{B}w(x)dx\right)\left(\frac{1}{|B|}\int_{B}w(x)^{-\frac{1}{p-1}}dx\right)^{p-1}<\infty,
 \end{align}
 where $r(B)$ denotes the radius of the ball $B$. We designate this class of weights as $A_{p, loc}(S)$. However, the Euclidean type $A_p$ condition is not necessary in the sense that there are weights $w$ for which $M$ is bounded on $L^p(w)$, but $w$ do not satisfy the Euclidean type $A_p$ condition (\ref{euclap}); see Example \ref{Apnot}. Nevertheless, following Ombrosi and Rivera-Ríos \cite{ORR}, a necessary condition can be obtained. Indeed, suppose that $w$ is a weight on $S$ for which $M$ is bounded on $L^p(w)$ with $1<p<\infty$. Then letting $\mathds{1}\otimes w$ stand for the standard product measure on $S\times S$,
 one can arrive at (see Section \ref{wlp}) 
 \begin{align}\label{neceapint}
     \mathds{1}\otimes w(\{(x,y)\in E\times F: d(x,y)<N\}) \leq C e^{2\rho N} w(E)^{\frac{1}{p}}w(F)^{1-\frac{1}{p}},
 \end{align}
 for all integers $N\geq 1$, and measurable sets $E, F\subseteq S$. It can be shown that the necessary condition (\ref{neceapint}) is not a sufficient condition for a weight, nor even for the weighted weak type $(p,p)$ estimates; see Example \ref{growthnec}. Nonetheless, assuming slightly less growth on $e^{2\rho }$, we give the following definition of $A_p$ weights, which turns out to be sufficient for weighted estimates.
 \begin{defn}\label{admap}(Admissible $A_p$ weights)
Let $1<p<\infty$, $\beta\in(0,1)$, $\beta\leq\alpha<p$. We say that a weight $w$ on $S$ is an admissible $A_p$ weight with parameters $(\alpha,\beta)$ if 
	\begin{enumerate}
		\item $w\in A_{p, loc}(S)$, i.e., $w$ satisfies (\ref{aploc}), and
		\item There exists a constant $C>0$ such that for every $N\in\N$, the following holds 
		$$\mathds{1}\otimes w(\{(x,y)\in E\times F: d(x,y)<N\})\leq C e^{2\rho\beta N} w(E)^{\frac{\alpha}{p}}w(F)^{1-\frac{\alpha}{p}}, $$
		for any pair of measurable sets $E, F\subseteq S.$ 
	\end{enumerate}
\end{defn}
\begin{rem}
We shall see in Section \ref{exmple} that the above class includes numerous examples of non-trivial weights. Interestingly, it includes certain positive spherical functions on $S$ or more generally Jacobi functions with specific parameters. To do so, we take advantage of the asymptotic behavior of Jacobi functions. We will discuss more on these in the next section.
\end{rem}
We now present the first main result of this article.
\begin{thm}\label{weightedp}
Let $1<p<\infty$, $\beta\in (0,1)$, $\beta\leq \alpha<p$. Suppose that $w$ is an admissible $A_p$ weight with parameters $(\alpha,\beta)$. Then the following statements are true.
\begin{enumerate}
    \item If $\beta<\alpha$, then 
    \begin{equation}\label{strongp}
        \|Mf\|_{L^p(w)}\lesssim\|f\|_{L^p(w)}.
    \end{equation}
    \item If $\beta=\alpha$, then
    \begin{equation}\label{weakp}
    \|Mf\|_{L^{p,\infty}(w)}\lesssim\|f\|_{L^p(w)}.
\end{equation}
\end{enumerate}
\end{thm}
\begin{rem}
    The theorem above is sharp in the sense that if $\beta=\alpha$, then there are examples of admissible $A_p$ weights with parameter $(\beta,\beta)$ for which weak type $(p,p)$ holds but the strong type $(p,p)$ does not hold; see Example \ref{notstrong}. Consequently, weak type $(q,q)$ also fails for all $q<p$.
\end{rem}
  We now state the second main result of this article, which is a variant of the Fefferman-Stein inequality (\ref{feffer}). 
  \begin{thm}\label{fsna}
 	Let $w$ be a weight on $S$ and $s>1$.  Then there exists a constant $C_s>0$ such that for all $\lambda>0$, we have 
	\begin{equation}\label{feffna}
		w(\{x\in S: Mf(x)>\lambda\})\leq \frac{C_{s}}{\lambda}\int_{S}|f(x)|M_sw(x)dx,
	\end{equation}
where $M_sw=(M(w^s))^{\frac{1}{s}}$. Moreover, $C_{s}\ra \infty$ as $s\ra 1$.
 \end{thm}
 As we shall see that the natural analogue of the Fefferman-Stein inequality does not hold in Harmonic $NA$ groups. In fact, we will show, by means of an example, that not only it is not possible to choose $s=1$ in (\ref{feffna}), but any number of iterations of the maximal operator is not sufficient for the validity of the Fefferman-Stein inequality. More generally, we have the following. 
\begin{thm}\label{thm_FS_not_1}
Let $k\in\N$ and $1\leq p<\infty $. Then there exists a weight $w$ on $S$ and a sequence $\{f_j\}\subset L^{p}(M^{(k)} w)$ such that 
\begin{equation}\label{eqn_not_FS}
     w\left( \left\{ M f_j >1\right\}\right) \geq c_j \int_S \left| f_j(x) \right|^p M^{(k)} w(x)\,dx,
\end{equation}
where $M^{(k)} = M\circ \overset{{k \text{ times}}}{\cdots} \circ M$, and $c_j \rightarrow \infty$, as $j\rightarrow \infty$.
\end{thm}
Fefferman and Stein used their inequality (\ref{feffer}) to prove vector-valued maximal inequalities. Using Theorem \ref{fsna}, we offer the following analogoue on Harmonic $NA$ groups. 
\begin{thm}\label{vectorvalued}
    Let $(f_1,f_2,f_3,...)$ be a sequence of functions on $S.$ Assume that $1<r\leq p<\infty.$ Then 
    $$\left\|\left(\sum_{n=1}^{\infty}Mf_n(\cdot)^r\right)^\frac{1}{r}\right\|_{L^p(S)}\lesssim_{p,r} \left\|\left(\sum_{n=1}^{\infty}|f_n(\cdot)|^r\right)^\frac{1}{r}\right\|_{L^p(S)}. $$
\end{thm}

\begin{rem}\label{introrem}
The following remarks are in order.
\begin{enumerate}
    \item Very recently, Antezana and Ombrosi \cite{AO} proved Theorem \ref{fsna} and Theorem \ref{weightedp} for real hyperbolic spaces. One can view real hyperbolic spaces as degenerate cases of Harmonic $NA$ groups \cite[p.210]{CDKR}. We remark that all the results proved in this article are true in this degenerate case as well.
    \item It is worth mentioning that taking, in particular, $w=1$ in Theorem \ref{weightedp}, and in Theorem \ref{fsna} respectively, we recover the classical results of Clerc-Stein \cite{CS} and Stromberg \cite{Str}  for rank one symmetric spaces of noncompact type respectively with different techniques (see Example \ref{trivial_ex}).
    \item We use, as one of the main ingredients of the proof of our results, an explicit dependence of the measure of the intersection of two geodesic balls in $S$ on the radii and the distance between two centers. More generally, this result is true for \textit{Gromov hyperbolic spaces}, see Proposition \ref{indira} for more details. This feature is very special to the underlying geometric structure of the manifold. In fact, Csik\'{o}s and Horv\'{a}th \cite[Theorem 3]{CH} proved that a connected, simply connected, complete Riemannian manifold is harmonic if and only if the volume of the intersection of two geodesic balls depends only on the distance between the centers and the radii of the balls. But, in 2006, Heber \cite{H} demonstrated that the only noncompact, homogeneous harmonic manifolds are the Euclidean spaces, rank one symmetric spaces, and Harmonic $NA$ groups. The discussion above implies, in particular, that higher rank symmetric spaces of noncompact type do not enjoy this property (viz. \eqref{indirac}).  This forces us to consider only the setting of Harmonic $NA$ groups and prevents us from extending the results of this article to higher rank symmetric spaces. 
\end{enumerate}
\end{rem} 
We end this section by briefly describing the plan of this paper. In the following section, we describe the relevant preliminaries on Harmonic $NA$ groups and record certain results necessary for our purpose.  In Section 3, we introduce the admissible $A_p$ class of weights formally and provide various examples. In the same section, we prove the weighted $L^p$-boundedness of $M.$ Lastly, we prove Fefferman-Stein type inequalities in the last section. The promised examples to back up our assertions about various features of the aforementioned theorems will be provided in the appropriate locations. 
\vspace{3mm}

\textbf{Notations.} The letters $\N$, $\Z$, $\R$, and $\C$ will respectively denote the set of all natural numbers, the ring of integers, and the fields of real and complex numbers. For $ z \in \C $, we use the notations $\Re z$ and $\Im z$ for real and imaginary parts of $z$, respectively. We shall follow the standard practice of using the letters $c$, $C$, $c_0$, $C_1$, etc., for positive constants, whose value may change from one line to another. Occasionally, the constants will be suffixed to show their dependencies on important parameters. Throughout this article, we use $X\lesssim Y$ or $Y\gtrsim X$ to denote the estimate $X\leq CY$ for some absolute constant $C>0$; if we need $C$ to depend on parameters, we indicate this by subscripts, that is, $X\lesssim_{\varepsilon}Y$ means that $X\leq C_{\varepsilon}Y$ for some constant $C_{\varepsilon}>0$ depending only on the parameter $\varepsilon$. We shall also use the notation $X\simeq Y$ for $X\lesssim Y$ and $Y\lesssim X$. For $1\leq p\leq\infty$, $p'$ denote the conjugate exponent of $p$, i.e., $1/p+1/p'=1$.

\section{Preliminaries}\label{prelim}
In this section, we first set up the notation and required preliminaries for the remaining parts of the paper. The following discussion on Harmonic $NA$ groups is standard and can be found, for instance, in \cite{ADY, DR1, CDKR}.

\subsection{Harmonic $NA$ groups}\label{har} Let $\mathfrak n$ be a two-step real nilpotent Lie algebra equipped with an inner product $\langle,\rangle$. Let $\mathfrak z$ be the center of $\mathfrak n$ and $\mathfrak v$ its orthogonal complement. We say that $\mathfrak n$ is an $H$-type algebra if for every $Z\in \mathfrak z$  the map $J_Z:\mathfrak v\to \mathfrak v$ defined by
\begin{equation*}
\langle J_Z X, Y\rangle=\langle [X, Y], Z\rangle,\ X, Y\in \mathfrak v,
\end{equation*}
satisfies the condition
$J_Z^2=-|Z|^2I_{\mathfrak v}$, $I_{\mathfrak v}$ being the identity operator on $\mathfrak v$. A connected and simply connected Lie group $N$ is called an $H$-type group or Heisenberg-type group if its Lie algebra is an $H$-type algebra. Since $\mathfrak{n}$ is nilpotent, the exponential map is a diffeomorphism, and hence we can parametrize elements of $N=\exp\mathfrak{n}$ by $(X,Z)$ for $X\in\mathfrak{v},\:Z\in\mathfrak{z}$. It follows from the Baker-Campbell-Hausdorff formula that the group law of $N$ is given by
$$(X,Z)(X',Z')=(X+X',Z+Z'+\frac{1}{2}[X,X']).$$
When $\mathfrak z=\R$, $\mathfrak v=\R^{2l}$ and for $s\in\R$, $J_s:\R^{2l}\to \R^{2l}$ is given by
\begin{equation*}
J_s(x,y)=(-sy,sx),\:\:\:\:\:x\in\R^l,\:y\in\R^{l},
\end{equation*}
then we get the Heisenberg group $H^l$ which is the prototype of a $H$-type group. We also note that the Lebesgue measure $dXdZ$ is a Haar measure on $N$.
The group $A=(0,\infty)$ acts on an $H$-type group $N$ by non-isotropic dilation:
\begin{equation}\label{nonisotropic}
\delta_a(n)=\delta_a(X,Z)=(\sqrt{a}X,aZ),\:\:a\in A,\:n=(X,Z)\in N.
\end{equation}
A Harmonic $NA$ group $S$ is the semi-direct product of a $H$-type group $N$ and $A$ under the above action. Thus, the multiplication on $S$ is given by
$$(X,Z,a)(X',Z',a')=(X+\sqrt{a}X',Z+aZ'+\frac{1}{2}\sqrt{a}[X,X'],aa').$$
Then $S$ is a solvable, connected, and simply connected Lie group having Lie algebra $\mathfrak{s}=\mathfrak{n}\oplus\mathfrak{z}\oplus\R$ with Lie bracket
$$[(X,Z,u),(X',Z',u')]=(\frac{1}{2}uX'-\frac{1}{2}u'X,uZ'-u'Z+[X,X'],0).$$
We write $(n,a)=(X,Z,a)$ for the element $\left(\exp(X+Z),a\right),\:a\in A,\:X\in\mathfrak{v},\:Z\in\mathfrak{z}.$ We remark that for any $Z\in\mathfrak{z}$ with $\|Z\|=1$, $J_Z^2=-I_{\mathfrak{v}}$ and hence $\mathfrak{v}$ is even dimensional. We suppose that $dim\:\mathfrak{v}=m,\:dim\:\mathfrak{z}=k$. Then $Q:=m/2+k$ is called the homogeneous dimension of $N$. We denote by $l$ the topological dimension $m+k+1$ of $S$. For convenience, we shall also use the notation $\rho=Q/2$. We note that $\rho$ corresponds to the half-sum of positive roots when $S=G/K$, is a rank one symmetric space of noncompact type. We denote by $e$ the identity element $(\underline{0},1)$ of $S$, where $\underline{0}$, $1$ are the identity elements of $N$ and $A$ respectively. The group $S$ is equipped with the left-invariant Riemannian metric $d$ induced by the inner product
$$\langle (X,Z,u),(X',Z',u')\rangle=\langle X,X'\rangle+\langle Z,Z'\rangle+uu'$$
on $\mathfrak{s}$. By $B(x,r)$ we mean the open ball centered at $x\in S$, and radius $r>0$ with respect to the metric $d$. The associated left Haar measure $dx$ on $S$ is given by
$$dx=a^{-Q-1}dXdZda,$$
where $dX,\:dZ,\:da$ are the Lebesgue measures on $\mathfrak{v},\:\mathfrak{z},\:A$ respectively. The left Haar measure or volume of the open ball $B(x,r)$ in $S$ satisfies the following estimate
\begin{equation}\label{ballvol}
     |B(x,r)|\simeq
     \begin{cases}
     r^{l},\hspace{0.5cm}\text{for all}\:\:r< 1\\
 e^{2\rho r},\hspace{0.5cm}\text{for all}\:\:r\geq 1.   
\end{cases}
\end{equation}
We mention that the right Haar measure on $S$ is $a^{-1}dXdZda$ and hence $S$ is not unimodular with the modular function $\Delta(X,Z,a)=a^{-Q}$. The functions which depend only on the distance from the identity element $e$ are called radial functions. Thus, a function $f_0$  on $[0,\infty)$ can be extended to a radial function $f$ on $S$ via the following obvious way
$$f(x)=f_0(d(e,x)),\,\,\,x\in S.$$
The convolution of two suitable functions $f$ and $g$ on $S$ is given by
$$f\ast g(x)=\int_Sf(y)g(y^{-1}x)\:dx=\int_Sf(xy)g(y^{-1})\:dx.$$
It is worth mentioning that the convolution is commutative for radial functions.
 Finally, we end the subsection by recording the so-called Kolmogorov inequality in our setting, which will be utilized in Section \ref{fsineq} while proving Theorem \ref{fsna}. Although the proof of this inequality is standard, we include it here for the sake of completeness. 
 \begin{prop}
 \label{kolmogorov}
     Let $T:L^1(S)\rightarrow L^{1,\infty}(S)$ be a bounded sublinear operator. Then for $f\in L^1(S)$, a ball $B$ in $S$, and $0<q<1$, we have 
     $$\int_{B}|Tf(x)|^qdx\leq \frac{\|T\|^q_{L^1\rightarrow L^{1,\infty}}}{1-q}|B|^{1-q}\|f\|_{L^1(S)}^p.$$
 \end{prop}
 \begin{proof}
     It is well-known that 
     $$\int_{B}|Tf(x)|^qdx=q\int_0^{\infty}\lambda^{q-1}|\{x\in B: |Tf(x)|>\lambda\}|d\lambda.$$
     Letting $K>0$ be a constant to be specified later, we decompose the above integral as 
     $$\int_0^K\lambda^{q-1}|\{x\in B: |Tf(x)|>\lambda\}|d\lambda+\int_K^{\infty}\lambda^{q-1}|\{x\in B: |Tf(x)|>\lambda\}|d\lambda ,$$
     which, in view of the hypothesis on $T$, is dominated by 
     \begin{align*}
         &|B|\int_0^K\lambda^{q-1}d\lambda+\|T\|^q_{L^1\rightarrow L^{1,\infty}}\|f\|_1\int_K^{\infty}\lambda^{q-2}d\lambda\\
         &= \frac{|B|}{q}K^{q}+\frac{\|T\|_{L^1\rightarrow L^{1,\infty}}}{1-q}\|f\|_1K^{q-1}.
     \end{align*}
     Now taking $K= |B|^{-1}\|f\|_1\|T\|_{L^1\rightarrow L^{1,\infty}}$, from the above observation, we obtain 
     $$\int_{B}|Tf(x)|^qdx\leq \left(1+\frac{q}{1-q}\right)\|T\|_{L^1\rightarrow L^{1,\infty}}^q\|f\|_1^q,$$
     from which, the result follows.
 \end{proof}
\subsection{Spherical functions.}\label{spherical} The spherical analysis on $S$ was initiated in \cite{DR1, Ric} and it is quite similar to that of the rank one symmetric space. In fact, it can be placed under the general framework of Jacobi analysis of Koornwinder \cite{Koo}. The spherical function $\vp_{\la}$ corresponding to $\la\in\C$, is the unique radial eigenfunction of the Laplace-Beltrami operator $\mathcal{L}$ on $S$ such that  
\begin{align*}
&\mathcal{L}\vp_{\la}=-(\la^2+\rho^2)\vp_{\la};\\
&\vp_{\la}(e)=1.
\end{align*}
The explicit expression of $\vp_{\la}$ can be found in \cite[(2.1), (2.34)]{ADY}. We now list down some properties of spherical functions which will be utilized later.
\begin{enumerate}
\item $\vp_{\la}=\vp_{-\la}$; 
\item $|\vp_{\la}|\leq\vp_{i\Im\la}$;
\item $\varphi_{\la}$ is a strictly positive function whenever $\la$ is purely imaginary or $\la=0$;
\item The following asymptotic behavior of spherical functions $\vp_{\la}$, whenever $\Im\la<0$, is known \cite[(2.7)]{ADY}
\begin{equation}\label{sph_asymp}
\lim_{d(x,e)\to\infty}e^{(-i\la+\rho)d(x,e)}\vp_{\la}(x)=\textbf{c}(\la),
\end{equation}
where $\textbf{c}(\la)$ is the Harish-Chandra \textbf{c}-function which has neither zero nor pole in the region $\Im\la<0$.
\end{enumerate}
We will now write spherical functions on $S$ as a particular case of Jacobi functions. A Jacobi function $\phi^{(\sigma,\tau)}_{\la}$, with $\sigma\geq\tau>-1/2$, $\la\in\C$, is defined by  
$$\phi^{(\sigma,\tau)}_{\la}(t)={}_2F_1\left(\frac{1}{2}(\sigma+\tau+1-i\la),\frac{1}{2}(\sigma+\tau+1+i\la);\sigma+1;-\sinh^2t\right),\hspace{0.3cm}t>0.$$
Spherical functions are related to Jacobi function in the following way
\begin{equation}\label{spjac}
    \vp_{\la}(x)=\phi^{(\sigma,\tau)}_{2\la}(d(e,x)/2),\hspace{0.3cm}x\in S,
\end{equation}
with indices $\sigma=(m+k-1)/2$, and $\tau=(k-1)/2$. For a detailed account of Jacobi functions we refer to \cite{Koo}.
For $\Im\la<0$, we have the following asymptotic of the Jacobi functions \cite[(2.19), p.8]{Koo}: 
\begin{equation}\label{jacasymp}
    \phi^{(\sigma,\tau)}_{\la}(t)=\textbf{c}_{(\sigma,\tau)}(\la)e^{(i\la-\varrho)t}(1+o(1)),\hspace{0.3cm}\text{as}\,\,t\to\infty,
\end{equation}
where $\varrho=\sigma+\tau+1$ and $\textbf{c}_{(\sigma,\tau)}(\la)$ is an analogue of Harish-Chandra  \textbf{c}-function which has neither zero nor pole in the region $\Im\lambda<0$. The Jacobi function $\phi^{(\sigma,\tau)}_{\la}$ is a smooth function on $[0,\infty)$ such that $\phi^{(\sigma,\tau)}_{\la}(0)=1$, and it satisfies the following differential equation
\begin{equation}\label{jacobi}
    \left(\frac{d^2}{dt^2}+\left((2\sigma+1)\coth{t}+(2\tau+1)\tanh{t}\right)\frac{d}{dt}+\la^2+\varrho^2\right)\phi^{(\sigma,\tau)}_{\la}(t)=0,\hspace{0.3cm}t\geq 0.
\end{equation}

For $\lambda \not = -i, -2i , \ldots,$ there is an  another solution $\Phi_{\lambda}^{(\sigma,\tau)}$ of \eqref{jacobi} on $(0,\infty)$ given by
\begin{align*}
   \Phi_{\lambda}^{(\sigma,\tau)}(t) = (2\cosh{t})^{i\lambda-\varrho} {}_2F_1\left( \frac{1}{2}(\varrho-i\lambda), \frac{1}{2}(\sigma-\tau+1-i\lambda);1-i\lambda;\cosh^{-2}t \right).
\end{align*}
Moreover, for $\lambda \in \C\setminus i \Z$, $\Phi_{\lambda}^{(\sigma,\tau)}$ and $\Phi_{-\lambda}^{(\sigma,\tau)}$ are two linearly independent solutions of \eqref{jacobi}. Therefore, $\phi^{(\sigma,\tau)}_{\lambda}$ can be written as a linear combination of $\Phi_{\lambda}^{(\sigma,\tau)}$ and $\Phi_{-\lambda}^{(\sigma,\tau)}$ in the following way
\begin{align*}
    \phi^{(\sigma,\tau)}_{\lambda}= \textbf{c}_{(\sigma,\tau)}(\lambda) \Phi_{\lambda}^{(\sigma,\tau)} + \textbf{c}_{(\sigma,\tau)}(-\lambda)\Phi_{-\lambda}^{(\sigma,\tau)}.
\end{align*}
We will use the following asymptotic estimates of $\Phi_{\lambda}^{(\sigma,\tau)} $, 
\begin{align}\label{Phi_inf}
    |\Phi_{\lambda}^{(\sigma,\tau)}(t) |\simeq e^{-(\Im\lambda+\varrho)t} , \quad \text{as $t\rightarrow \infty$,}
\end{align}
and also the following limiting behavior \cite[(8.12)]{SP}
\begin{align}\label{Phi_zero}
    |\Phi_{\lambda}^{(\sigma,\tau)} (t)| \simeq \begin{cases}
        t^{2\sigma} \quad\text{if $\sigma \not=0$}\\
        \log \frac{1}{t} \quad\text{if $\sigma =0$}
    \end{cases}\text{as}~~t\rightarrow 0.
\end{align}
\subsection{Gromov hyperbolic spaces}\label{geo} We now discuss a result, alluded to in the introduction, due to Chatterji and Niblo \cite{CN} regarding an estimate of the measure of the intersection of two balls. Before this, we recall the following definition from metric geometry \cite[Definition 4.2]{Kni}. 
\begin{defn}
    Let $\delta\geq 0$. A geodesic metric space is called $\delta$-hyperbolic if each side of a geodesic triangle is contained in the $\delta$-neighborhood of the two other sides. A geodesic metric space is called Gromov hyperbolic is a geodesic metric space if it is $\delta$-hyperbolic for some $\delta\geq 0$.
\end{defn}

As $|B(x,r)|\simeq e^{2\rho r}$, for $r\geq1$, the Harmonic $NA$ group $S$ is of purely exponential volume growth in the sense of Knieper \cite[Definition 2.2]{Kni} and hence $S$ is Gromov hyperbolic (see \cite[Theorem 4.8]{Kni}). In the proof of \cite[Proposition 14]{CN}, Chatterji and Niblo have shown the following which we put in the form of a lemma.
\begin{lem}
 Let $(X,\Tilde{d})$ be a $\delta$-hyperbolic space. Then for any $x,\,y\in X$, and $s\,t>0$
 $$B(c,r)\subseteq B(x,s)\cap B(y,t)\subseteq B(c,r+2\delta),$$
 where $r=\frac{1}{2}(s+t-\Tilde{d}(x,y))$, and $c$ is any point on any geodesic between $x$ and $y$, at distance $\frac{1}{2}(s-t+\Tilde{d}(x,y))$ from $x$.
\end{lem}
As an immediate consequence, we have the following estimate of the measure of the intersection of two balls.
\begin{prop}\label{indira}
For $x,\,y\in S$ and $s,\,t>0$, we have
\begin{equation}\label{indirac}
|B(x,s)\cap B(y,t)|\simeq e^{\rho(s+t-{d}(x,y))}.  \end{equation}
\end{prop}

  \section{Weighted $L^p$-boundedness}\label{wlp}
 
 In this section, we prove weighted $L^p$-boundedness for the centered Hardy-Littlewood maximal operator with weights from the admissible $A_p$ class for $S$, which will be presented here more rigorously. The methods are more or less based on the ideas developed in \cite{ORR,AO}. Additionally, we give a variety of examples of weights and show how they are used to demonstrate different facets of the results we prove. Let us start by recalling a few terminologies. 
 
For each $r>0$, let $A_r$ stand for the ball averaging operator, which, when written in terms of the group convolution, takes the form
	\begin{equation*}
	    A_rf(x)=\frac{1}{|B(x,r)|}\int_{B(x,r)}|f(y)|\,dy=\frac{1}{|B(x,r)|} |f|\ast\mbox{\large$\chi$}_{B(e,r)}(x),~~x\in S.
     \end{equation*}
 Recall that the local and large-scale variants of $M$ given by 
	$$M^{0}f(x)=\sup_{0<r\leq T}A_rf(x)\quad\quad  \text{and} \quad\quad M^{\infty}f(x)=\sup_{r> T}A_rf(x),$$
 where $T>0$ is a finite number. We begin by noting that the $L^p$-mapping properties of the aforementioned operators do not depend on the choice of $T$. This follows easily from the fact that for any $0<t_0<t_1$, we have 
 \begin{equation*}
     \sup_{t_0<r<t_1}A_rf\leq \frac{1}{|B(e,t_0)|}|f|\ast\mbox{\large$\chi$}_{B(e,\ t_1)},
     \end{equation*}
 and $\mbox{\large$\chi$}_{B(e,t_1)}$ is an integrable function on $S.$ Nevertheless, for our convenience, we take $T=2$ throughout this article. 
 Now, as already mentioned in the introduction, for the weighted $L^p$-boundedness of the local variant $M^0$, one needs to assume that the weights have to satisfy the $A_{p, loc}(S)$ condition (\ref{aploc}). In order to handle the other part, namely $M^{\infty}$, we first proceed to see what condition on weights is necessary for $M$ to be bounded on weighted $L^p$ spaces. Let $w$ be a weight on $S$ such that $M$ is bounded on $L^p(w)$ with $1<p<\infty$. Then from the definition of product measure, it is clear that
 \begin{align*}
     \mathds{1}\otimes w(\{(x,y)\in E\times F: d(x,y)<N\}) =\int_Fw(E\cap B(x,N))\,dx.
 \end{align*}
 The integral on the right-hand side  of the above equation is equal to
 \begin{align*}
     &|B(e, N)|\int_{F}A_{N}(\Chi_{E})(x)w(x)\,dx\\
     & \lesssim e^{2\rho N} \int_{F} M(\Chi_E)(x)w(x)\,dx\\
     & \leq e^{2\rho N} \left(\int_{F} M(\Chi_E)^p(x)w(x)\,dx\right)^{\frac1p}\left(\int_{F} w(x)\,dx\right)^{\frac1{p'}}\\
     &\leq e^{2\rho N} \|M(\Chi_E)\|_{L^p(w)} w(F)^{p'},
 \end{align*}
 where in the second last step we have used H\"older's inequality. Now, using the $L^p(w)$-boundedness of $M$, it follows that
 \begin{align}\label{neceap}
     \mathds{1}\otimes w(\{(x,y)\in E\times F: d(x,y)<N\}) \lesssim e^{2\rho N} w(E)^{\frac{1}{p}}w(F)^{1-\frac{1}{p}},
 \end{align}
 for all integers $N\geq 1$, and measurable sets $E, F\subseteq S$.  However, as we shall see later in this section, the condition above is not sufficient for proving the weighted $L^p$-boundedness of $M$; see Example \ref{growthnec}. Nevertheless, assuming slightly less growth on $e^{2\rho}$, we define $A_p$ weights as follows, which proves to be sufficient for weighted estimates. 
 \subsection{Admissible $A_p$ weights}In what follows, we fix $1<p<\infty.$ 
 \begin{defn}\label{defn_A_p}(Admissible $A_p$ weight)
Let $\beta\in(0,1)$, $\beta\leq\alpha<p$. We say that a weight $w$ is an admissible $A_p$ weight with parameters $(\alpha,\beta)$ if 
	\begin{enumerate}
		\item $w\in A_{p, loc}(S)$, i.e., 
		$$\sup_{r(B)\leq 2}\left(\frac{1}{|B|}\int_{B}w(x)dx\right)\left(\frac{1}{|B|}\int_{B}w(x)^{-\frac{1}{p-1}}dx\right)^{p-1}<\infty,$$ 
		\item (Large-scale condition)  for every $N\in\N$, the following holds 
		$$\mathds{1}\otimes w\left( \left\{ (x,y)\in E\times F: d(x,y)<N\right\} \right) \lesssim e^{2\beta \rho N} w(E)^{\frac{\alpha}{p}}w(F)^{1-\frac{\alpha}{p}} $$
		for any pair of measurable sets $E,F\subseteq S.$ 
	\end{enumerate}
\end{defn}
After a moment's thought staring at the above definition, it may appear that the large-scale condition can be challenging to test for a given weight. We, therefore, offer adequate requirements for weights to fall within the admissible $A_p$ class with specific parameters. It will be then clear that the size of the class of admissible $A_p$ weights is rather substantial. In this context, we first note from the calculation above that 
\begin{align}
\label{pm-int}
         \int_{F}A_{N}(\Chi_{E})(x)w(x)dx=\frac{1}{|B(x,N)|}\mathds{1}\otimes w(\{(x,y)\in E\times F: d(x,y)<N\}).
     \end{align}         
    This interesting observation allows us to restate the large-scale condition in terms of the averaging operator, which is very useful in certain cases. Given $j\in \mathbb{N}$, we write $$\Omega_j:=B(e,j)\setminus B(e,j-1).$$ Thus, $S=\cup_{j=1}^\infty\Omega_j $ provides a decomposition of $S$ into annular regions, which will be beneficial in estimating several things in our proofs. It follows from Proposition (\ref{indira}) that 
    \begin{equation}\label{annular}
        |\Omega_j\cap B(x,N)|\simeq e^{\rho(N+j-d(e,x))},
    \end{equation}
 for all $x\in S$; $j,\, N\in\N$.   We now provide our first criteria.
\begin{prop}\label{easy_check}
    Let $w$ be a weight on $S$. Assume that there exists $\eta<1$ such that for any $i,j,N\in\N$ with $|i-j|\leq N $, we have
    \begin{equation}\label{check_examp}
        w\left(  \Omega_i\cap B(x,N)\right)\leq C e^{\rho (N+i-j)(p-\eta)}e^{2\rho N\eta} w(x),
    \end{equation}
    for almost every $x\in \Omega_j$. Then $w$ satisfies the large-scale condition in the definition of admissible $A_p$ class with parameters $(\frac{p}{p+1-\eta}, \frac{p}{p+1-\eta}).$
\end{prop}
\begin{proof}
    In view of \eqref{pm-int}, given a pair of measurable subsets $E, F$ of $S$, we begin with estimating the following integral 
    $$I:=\int_{F}A_{N}(\Chi_{E})(x)w(x)dx.$$
    In order to do that, we decompose 
     $$I= \sum_{i,j\geq 1}\int_{F_i}A_N(\Chi_{E_j})(x)w(x)dx$$ where $E_j$ and $F_i$ are defined by  
     $$F_i=F\cap \Omega_i,~~~~E_j=E\cap \Omega_j.$$
     Now, denoting the integrals inside the last sum by $I_{j,i},$ we note that 
     \begin{align}\label{Ijidef1}
         I_{j,i}=\frac{1}{|B(e,N)|}\int_{F_i}\left(\int_{E_j\cap B(y,N)}dx\right)w(y)dy.
     \end{align}
     But using the Proposition \ref{indira}, we see that, given $y\in F_i$, 
     $$|E_j\cap B(y,N)|\lesssim  e^{\rho (N+j-i)},$$
which then yields 
     \begin{align}\label{Iji1}
         I_{j,i}\lesssim e^{-2\rho N}e^{\rho(N+j-i)}w(F_i).
     \end{align}
     Again, following a simple change of order of integration in \eqref{Ijidef1}, we see that 
     \begin{align*}
         I_{j,i}=\frac{1}{|B(e,N)|}\int_{E_j}\left(\int_{F_i\cap B(x,N)}w(y)dy\right)dx.
     \end{align*}
     Then, using the hypothesis, we get the following estimate 
     \begin{align*}
         I_{j,i} \lesssim e^{-2\rho N}e^{\rho(N+i-j)(p-\eta)}e^{2\rho N\eta} w(E_j),
     \end{align*}
     whence we obtain 
     \begin{align*}
         I_{j,i}\lesssim e^{-2\rho N} \min \left(e^{\rho (N+j-i)}w(F_i),e^{\rho (N+i-j)(p-\eta)}e^{2\rho N\eta} w(E_j)\right).  
     \end{align*}
     Therefore, 
     \begin{align*}
         I \lesssim  e^{-2\rho N}\sum_{i,j} \min \left(e^{\rho (N+j-i)}w(F_i),e^{\rho (N+i-j)(p-\eta)}e^{2\rho N\eta}w(E_j)\right).
     \end{align*}
     In order to estimate the above sum, we take some $\varepsilon\in\R$ to be specified later and decompose the above sum into two parts: $i<j+\varepsilon$ and $i\geq j+\varepsilon$ which allow us to dominate the concerned sums by 
     \begin{align*}
        &e^{2\rho N\eta} e^{\rho N(p-\eta)} \sum_{\substack{i,j\geq0\\ i<j+\varepsilon}} e^{\rho (i-j)(p-\eta)} w(E_j) +\sum_{\substack{i,j\geq0\\ i\geq j+\varepsilon}} e^{\rho (N+j-i)}w(F_i)\\
        &\leq e^{2\rho N\eta} e^{\rho N(p-\eta)}\sum_{j\geq0}\sum_{i<j+\varepsilon}e^{\rho (i-j)(p-\eta)} w(E_j)+ \sum_{i\geq0} \sum_{j\leq i-\varepsilon}e^{\rho (N+j-i)}w(F_i)\\
        &\leq e^{2\rho N\eta} e^{\rho N(p-\eta)} \sum_{j\geq0} e^{\rho (j+\varepsilon-j)(p-\eta)} w(E_j)+ e^{\rho N}\sum_{i\geq0}e^{\rho (i-\varepsilon-i)}w(F_i)\\
        &\leq e^{2\rho N\eta} e^{\rho N(p-\eta)}e^{\rho \varepsilon(p-\eta)} w(E)+ e^{\rho N}e^{-\rho \varepsilon} w(F)
     \end{align*}
     which, after choosing $$\varepsilon= \frac{2\log_{a}((w(F)/w(E))}{p-\eta+1}-N\frac{p+\eta-1}{p-\eta+1}$$ with $a=e^{2\rho }$, is dominated by a constant (depending on $p$ and $\eta$) multiple of $$e^{2\rho N\frac{p}{p+1-\eta}}w(E)^{\frac{1}{p+1-\eta}}w(F)^{1-\frac{1}{p+1-\eta}}.$$
     Hence, we see that 
     \begin{align*}
         &\mathds{1}\otimes w(\{(x,y)\in E\times F: d(x,y)<N\}) \\
         &=|B(e,N)|\int_{F}A_{N}(\Chi_{E})(x)w(x)dx\\
         &\lesssim e^{2\rho N\frac{p}{p+1-\eta}}w(E)^{\frac{1}{p+1-\eta}}w(F)^{1-\frac{1}{p+1-\eta}},
     \end{align*}
     proving the proposition.
\end{proof}
The previous result illustrates how, given a weight, we may examine the large-scale condition by observing the behavior of the weight in annular regions. However, the one limitation we do have is that this can only be used to provide examples of admissible $A_p$ weights with specific parameters $(\alpha, \beta)$ such that $\alpha=\beta.$ So, to accommodate the case $\beta<\alpha$, we provide the following result. For this purpose, we introduce a discrete variant of maximal functions defined by 
\begin{align}
    M^{dis}g=\sup_{N\in \mathbb{N}}\ A_Ng,~\ M_s^{dis}g=M^{dis}(g^s)^{\frac1s}.
\end{align}
\begin{prop}\label{betalessalpha}
    Let $w$ be a weight on $S$, and let $s>1$ be such that
   \begin{align}\label{eqn_Msw<w}
       M^{dis}_s (w) \lesssim w.
   \end{align}Then for each $p\in(1,\infty)$, $w$ satisfies the large-scale condition in the definition of admissible $A_p$ class with parameters $\left(\frac{s'p}{s'+1}, \frac{s'}{s'+1}\right)$.
\end{prop}
This follows from the next lemma, which provides the weighted behavior of the averaging operator locally.   
\begin{lem}
 \label{weak-l1-lem1}
	    Let $w$ be a weight and $s>1$. Then, given $N\in\mathbb{N}$ and a pair of measurable subsets $E,F$ of $S$, we have 
     \begin{equation}\label{eqn_A_N_FS}
         \int_{F}A_N(\Chi_E)(x)w(x)\,dx\lesssim e^{-2\rho \frac{N}{s'+1}}w(F)^{\frac{1}{s'+1}}M_sw(E)^{\frac{s'}{s'+1}}.
     \end{equation}
	\end{lem}
 \begin{proof}
     Keeping in mind the symbols employed in the proof of the Proposition \ref{easy_check}, we have 
     $$I_{j,i}\lesssim e^{-2\rho N}e^{\rho(N+j-i)}w(F_i).$$ 
     On the other hand, following a simple change of order of integration in (\ref{Ijidef1}), we can rewrite
     \begin{align*}
         I_{j,i}=\frac{1}{|B(e,N)|}\int_{E_j}\left(\int_{F_i\cap B(x,N)}w(y)\,dy\right)dx.
     \end{align*}
     Applying H\"older's inequality with exponent $s$, $I_{j,i}$ is dominated by 
     \begin{align*}
              \int_{E_j}\left(\int_{F_i\cap B(x,N)}dy\right)^{\frac{1}{s'}}\left(\int_{B(x,N)}w(y)^sdy\right)^{\frac1s}dx\lesssim e^{\rho (N+j-i)/s'}M^{dis}_s(w)(E_j),
     \end{align*}
     where in the last inequality we have used the estimate  $$|F_i\cap B(x,N)|\lesssim e^{\rho (N+i-j)}$$ whenever $x\in E_j$ (see Proposition \ref{indira}).  Therefore, 
     \begin{equation}
     \label{Iji2}
         I_{j,i}\lesssim  e^{\rho (N+j-i)/s'}M^{dis}_s(w)(E_j).
     \end{equation}
     So, combining the above two estimates of $I_{j,i}$, we obtain 
     $$I\lesssim  e^{-2\rho N}\sum_{i,j\geq 1} \min \left( e^{\rho (N+j-i)}w(F_i), e^{\rho (N+i-j)/s'}M^{dis}_s(w)(E_j)\right).$$
     At last, proceeding as in the proof of the Proposition \ref{easy_check}, we have 
     $$I \lesssim e^{-2\rho \frac{N}{s'+1}}w(F)^{\frac{1}{s'+1}}M^{dis}_sw(E)^{\frac{s'}{s'+1}}, $$
     completing the proof of this lemma.
 \end{proof}
 \textbf{\textit{Proof of the Proposition \ref{betalessalpha}}}: In view of \eqref{pm-int}, using the above lemma, we conclude that 
 $$\mathds{1}\otimes w\left( \left\{ (x,y)\in E\times F: d(x,y)<N\right\} \right) \lesssim e^{2\rho{N}\left( 1-\frac{1}{s'+1} \right)}w(F)^{\frac{1}{s'+1}}M^{dis}_sw(E)^{\frac{s'}{s'+1}},$$
 for any two measurable subsets $E$ and $F$ of $S.$ Now, the hypothesis on $w$, allows us to dominate the right-hand side of the above inequality by 
 $$ C_{s}e^{2\rho{N}\left(\frac{s'}{s'+1} \right)}w(E)^{\frac{s'}{s'+1}}w(F)^{\frac{1}{s'+1}},$$
 proving that the weight $w$ satisfies the large-scale condition with parameters $\alpha=\frac{s'}{s'+1}p$, and $\beta=\frac{s'}{s'+1}.$\qed

Now that we have the necessary apparatus, we can give various examples of weights in the following subsection.
 \subsection{Examples}\label{exmple}
 \begin{enumerate}
     \item\label{trivial_ex} \textit{(Trivial example)} We denote by $w_0$, the weight which is identically $1$. Then $w_0$  belongs to the class of admissible $A_p$ weights with parameters $\alpha=\frac{s'}{s'+1}p$, and $\beta=\frac{s'}{s'+1}$, for any $1<p<\infty.$
 Indeed, it is easy to see that $w_0\in A_{p, loc}(S).$ Now, to check the large-scale condition, recalling  the definition of $M_s$, we see that $$M_sw_0(x)=\left(\sup_{r>0}\frac{1}{|B(x,r)|}\int_{B(x,r)}1^sdx\right)^{\frac{1}{s}}=1=w_0(x), $$
 whence by the Proposition \ref{betalessalpha}, we are done.
 \vspace{.3cm}
 \item\label{blesa}  \textit{(Case: $\beta<\alpha$)} For $\gamma\in [-1, 0)$, we define  
 $$w_{\gamma}(x)=e^{2\rho\gamma d(e,x)},\hspace{0.3cm}x\in S.$$ 
 We  show that $w_\gamma$ is in the class of admissible $A_p$ weights with parameters $$\alpha=\frac{s'}{s'+1}p,\hspace{0.3cm} \beta=\frac{s'}{s'+1},$$ 
 for some $s>1$ depending on $\gamma$.  Now, as $w_{\gamma}^s=w_{\gamma s},$ to check the large-scale condition, in view of the Proposition \ref{betalessalpha}, it is enough to find some $s>1$ such that $\gamma s\in [-1,0)$ and $M^{dis}w_{\gamma s}\lesssim w_{\gamma s}$. This is achieved once we prove the following 
 \begin{equation}
 \label{claimwgama}
     M^{dis}w_{\gamma}\lesssim w_{\gamma},\hspace{0.3cm}\text{for all}\:\:\gamma\in [-1,0).
 \end{equation}
 If  $x \in \Omega_j$, then it can be shown that the set $B(x,N) \cap \Omega_i$ is empty unless $$|i-j| \leq N+1,$$ or equivalently
 \begin{equation}\label{i_j_N}
     i =j+N+1-2m, \quad \text{for some } m\in \{ 0,1,\ldots N+1\} .
 \end{equation}
 This follows from the observation that if  $ y\in B(x,N) \cap \Omega_i$, then we necessarily have  
\begin{align}\label{eqn_d(0,y)}
    i-1 \leq d(e,y) \leq d(x,y) + d(e,x)  \leq N+j.
\end{align}
Similarly, interchanging the role of $i$ with $j$, and $x$ with $y$ we get the other inequality.

Now, in order to prove \eqref{claimwgama}, it is enough to prove for all $N\in\mathbb{N}$ 
\begin{align*}
    \frac{1}{|B(x,N)|} \int_{B(x,N)} w_{\gamma}(y)\, dy\lesssim w_{\gamma}(x),~~a.e.\ x\in S.
\end{align*}
 Since for  $y \in \Omega_k$, $w_{\gamma}(y) = e^{2\rho \gamma d(e,y) } \simeq e^{2\rho \gamma k}$, for any $k\in \N$, we can write
 \begin{align*}
      &\frac{1}{|B(x,N)|} \int_{B(x,N)} w_{\gamma}(y) \,dy\\
      & =  \frac{1}{|B(x,N)|} \sum_{i=1}^{\infty}\int_{\Omega_i} \Chi_{B(x,N)}(y) w_{\gamma}(y) \,dy\\
      &= \frac{1}{|B(x,N)|} \sum_{m=0}^{N+1}\int_{\Omega_{j+N-2m}} \Chi_{B(x,N)}(y) w_{\gamma}(y) \,dy\quad \text{(as $x\in\Omega_j$, we use \eqref{i_j_N})}\\
      &\simeq \frac{1}{|B(x,N)|} \sum_{m=0}^{N+1}\int_{\Omega_{j+N-2m}} \Chi_{B(x,N)}(y) e^{2\rho \gamma(j+N-2m) } \,dy\\
      & = \frac{e^{2\rho \gamma (j+N)}}{|B(x,N)|} \sum_{m=0}^{N+1}  |{B(x,N)\cap \Omega_{j+N-2m}}|e^{- 4\rho \gamma m }.\\
 \end{align*}
 Using the estimate of  $|{B(x,N)\cap \Omega_{j+N-2m}}|$ from \eqref{annular}, we get
 \begin{align*}
      \frac{1}{|B(x,N)|} \int_{B(x,N)} w_{\gamma}(y) \,dy& \simeq \frac{e^{2\rho \gamma (j+N)}}{|B(x,N)|}\sum_{m=0}^{N+1}  e^{\rho (N+j+N-2m-j)} e^{- 4\rho \gamma m }\\
      & = \frac{e^{2\rho \gamma (j+N)}}{|B(x,N)|}\sum_{m=0}^{N+1}  e^{\rho (2N-2m)} e^{- 4\rho \gamma m }\\
       & \simeq {e^{2\rho \gamma (j+N)}}\sum_{m=0}^{N+1}    e^{- 2\rho (2\gamma+1) m },
 \end{align*}
 for all $x\in \Omega_j$. 
 We first prove \eqref{claimwgama} for the case $2\gamma +1 > 0$, i.e., $-1/2\leq \gamma<0$. We recall that $w_{\gamma}(x) \simeq e^{2\rho \gamma j}$, for $x \in \Omega_j$, and so we have
     \begin{align*}
    \frac{1}{|B(x,N)|} \int_{B(x,N)} w_{\gamma}(y)\, dy&\simeq 
{e^{2\rho \gamma (j+N)}} \sum_{m=0}^{N+1}    e^{- 2\rho (2\gamma+1) m } \\
&\lesssim  e^{2\rho \gamma N} \frac{(N+1)(N+2)}{2} w_{\gamma}(x)\\
&\lesssim_{\gamma}w_{\gamma}(x),
 \end{align*}
 as $-1/2<\gamma<0$. This completes our proof of \eqref{claimwgama} for $\gamma\in [-1/2,0)$. We now take $-1 \leq 2\gamma+1 <0$, or equivalently $-1\leq \gamma <-1/2 $, then 
 \begin{align*}
     e^{2\rho \gamma N}\sum_{m=0}^{N+1}    e^{- 2\rho (2\gamma+1) m } =  e^{2\rho \gamma N}  \frac{ e^{- 2\rho (2\gamma+1) (N+2)}-1}{ e^{- 2\rho (2\gamma+1)  }-1}\simeq e^{-2N(\gamma+1)}. 
 \end{align*}
This concludes the proof of \eqref{claimwgama}.
\vspace{0.3cm}
\item \label{betaeqalpha} (\textit{Case: $\beta=\alpha$})  
 Let $$w_{p-1}(x)=e^{2\rho(p-1)d(e,x)},\hspace{0.3cm} x\in S.$$ We show that $w_{p-1}$ belongs to the admissible $A_p$-class with parameters $\alpha=\beta=1/2.$ Indeed, once again it is obvious that $w_{p-1}\in A_{p, loc}(S).$ To check the large-scale condition, let us fix $N\in \mathbb{N}$ and take $i,j \in \N \cup \{0\} $  with  $|i-j|\leq  N.$ Then for $x\in \Omega_j$, we write
\begin{align*}
    w(\Omega_i \cap B(x,N)) = \int_S \Chi_{\Omega_i \cap B(x,N)}(y) w (y) dy \lesssim e^{2\rho (p-1) i}   |\Omega_i \cap B(x,N)|. 
\end{align*}
Using the estimate of $|\Omega_i \cap B(x,N)|$ from \eqref{annular},  we  obtain,
\begin{align*}
     w (\Omega_i \cap B(x,N)) \lesssim  e^{2\rho (p-1) i} e^{\rho(N+i-j)}
\end{align*}
Rewriting the right-hand side as 
$$e^{\rho (2p-1) (N+i-j)} e^{-\rho(2p-1) N -\rho (i-j)}  e^{\rho(N+i-j)}e^{2\rho(p-1)j},$$ and simplifying in view of the fact that $w_{p-1}(x)\simeq e^{2\rho(p-1)j},$ for $x\in \Omega_j$, we get 
$$w (\Omega_i \cap B(x,N))  \lesssim e^{\rho (2p-1) (N+i -j) } e^{2\rho N(1-p)} w_{p-1} (x).$$
Thus, $w$ satisfies \eqref{check_examp} hypothesis of Proposition \ref{easy_check} with $\eta =1-p$. Therefore, $w_{p-1}$ belongs to the admissible $A_p$-class with parameters $\alpha=\beta= 1/2.$
\vspace{.3cm}

\item (\textit{Spherical functions as examples}) This case is devoted to showing certain spherical functions or more generally some Jacobi functions considering radial functions on $S$, constitute examples of admissible $A_p$ weights.  Keeping in mind the notations introduced in Section \ref{spherical},  let us define $$u(t) = \phi^{(\sigma,\tau)}_{i\kappa}(t), \quad \text{for } t\geq 0,$$ where $ \kappa = 2\rho (p-1) +\varrho$.
We claim that $u$ belongs to $A_p$ class with parameters $\alpha=\beta =1/2$. To see this, we recall from \eqref{jacasymp},
\begin{align*}
    u(t) \simeq e^{(\kappa -\varrho)t}= e^{2\rho(p-1)t},
\end{align*}
whence from the above case \eqref{betaeqalpha}, the claim follows. 
Let $$v(t)= \frac{t^{2\sigma}}{(1+t^{2\sigma})}\Phi^{(\sigma,\tau)}_{i\theta}(t) ,\,\,\,t>0,$$ where $\theta =-2\rho \gamma -\varrho$, and $\gamma\in[-1/2,0)$. Using the asymptotic estimates \eqref{Phi_inf} and \eqref{Phi_zero} of $\Phi^{(\sigma,\tau)}_{i\theta}$, it follows that
\begin{equation*}
    v(t) \simeq e^{-(\theta +\varrho)t}\simeq e^{2\rho\gamma t}.
\end{equation*} Then, from the case \eqref{blesa}, it follows that  $v$ is an admissible $A_p$ weight with parameters $$\alpha=\frac{s'}{s'+1}p,\hspace{0.3cm} \beta=\frac{s'}{s'+1},$$
for some $s >1$ (depending on $\gamma$). As mentioned in Section \ref{spherical}, any spherical function of $S$ is a  Jacobi function with specific parameters. The above discussion reveals that for every $p\in(1,\infty)$ there is a spherical function which is an admissible $A_p$ weight with parameters $(1/2,1/2)$. More generally, one can consider the following class of weights of the form 
\begin{equation}\label{generalweight}
    w(x) =\varphi_{i\gamma} (x) \eta(x),\,\,\,x\in S
\end{equation}
where $\gamma =(2p-1)\rho$, and $\eta$ is any positive bounded function with non-zero infimum. For example, one can take  $\eta (x) =e^{\frac{1}{1+d(e,x)}}.$ This kind of weight has been used in Dahlner \cite[p.8]{Dalhner} in the case of rank one symmetric spaces of noncompact type.
 \end{enumerate}
  
 \subsection{Weighted inequalities}

 \begin{lem}\label{lem_A_1_leb_est}
     Let $\beta\in(0,1)$, $\beta\leq\alpha<p$ and $w$ be an admissible $A_p$ weight with parameters $(\alpha, \beta)$. Then for each $\lambda>0$, $N\in \mathbb{N}$, there exists $\eta>0$ such that  
     \begin{equation}\label{A_1_f_est}
         w\left(\{A_N(A_1f)\geq \lambda\}\right) \lesssim  \sum_{k=0}^{N} e^{\rho (k-N)(1-\beta)\frac{p}{\alpha}} e^{2\rho \beta\frac{p}{\alpha}k}~w\left(\{A_2f> \eta\lambda e^{2\rho k}\}\right).
     \end{equation}
 \end{lem}
 \begin{proof} Without loss of generality, we may assume that $f$ is non-negative. In order to deal with the exponential volume growth property of $S$, we decompose $A_1f$ as follows:
     $$A_1f=f\Chi_{E_{-1}}+\sum_{k=0}^NA_1f\Chi_{E_k}+A_1f\Chi_{E_{\infty}},$$ where for each non-negative integer $k$, $E_k$ is the sub-level sets for $A_1f$ defined by  
    \begin{align*}
        E_k& =\{x\in S: e^{2\rho (k-1)}\leq A_1f(x)<e^{2\rho k}\};\\
        E_{-1}&=\{x\in S: A_1f(x)\leq e^{-2\rho }\};\\
        E_{\infty}&= \{x\in S: A_1f(x)\geq e^{2\rho N}\}.
    \end{align*}
     Thus, we obtain the following bound
     $$A_1f\leq e^{-1}+\sum_{k=0}^Ne^{2\rho k}\Chi_{E_k}+A_1f\Chi_{E_{\infty}}$$ which, for any $N\in \mathbb{N}$, yields 
     \begin{equation*}
         A_N(A_1f)\leq e^{-1}+\sum_{k=0}^Ne^{2\rho k}A_N(\Chi_{E_k})+A_N(A_1f\Chi_{E_{\infty}}).
     \end{equation*}
      The inequality above allows us to observe that 
     \begin{equation}
     \label{lp-lem1eq1}
         w\left(\left\{A_N(A_1f)\geq 1\right\}\right) \leq w\left(\left\{ \sum_{k=0}^Ne^{2\rho k}A_N(\Chi_{E_k})\geq e^{-1}\right\}\right)+ w\left(\{A_N(A_1f\Chi_{E_{\infty}})\geq 1-2e^{-1}\}\right).
     \end{equation}
     We now estimate the last term in the expression above. 
    In order to do so, we first notice that  
    \begin{equation}
        \label{lp-eq-2}
        w\left(\{A_N(A_1f\Chi_{E_{\infty}})\neq 0\}\right)\leq w\left(\{x\in S: B(x, N)\cap E_{\infty}\neq \emptyset\}\right).
    \end{equation}
    We claim that 
     \begin{equation}\label{claim_A_N+1}
         w\left(\{x\in S: B(x, N)\cap E_{\infty}\neq \emptyset\}\right)\leq w\left(\left\{ A_{N+1}(\Chi_{\{A_2f\geq c_0e^{2\rho N}\}})\geq c_0e^{-2\rho N} \right\}\right)
     \end{equation}
     for some constant $c_0>0.$ To substantiate this claim, we observe that for any $y\in B(x, N)\cap E_{\infty}$, we have $$B(y,1)\subset B(x, N+1)\cap \{z\in S: A_2f(z)\geq c_0e^{2\rho N}\}.$$ Indeed, given such $y$, letting $z\in B(y,1)$, we note that 
     \begin{align*}
         A_{2}f(z)=\frac{1}{|B(z,2)|}\int_{B(z,2)}f(z')\,dz'\geq \frac{1}{|B(z,2)|}\int_{B(y,1)}f(z')\, dz'=\frac{|B(y,1)|}{|B(z,2)|}A_1f(y),
     \end{align*}
     as $B(y,1)\subset B(z,2)$. Now, as $y\in E_\infty$, we see that $$A_2(z)\geq c_1e^{2\rho N},$$ where $c_0:=\frac{|B(y,1)|}{|B(z,2)|}.$ But then, for any $x$ with $B(x, N)\cap E_{\infty}\neq \emptyset$, we obtain 
     \begin{align*}
         A_{N+1}(\Chi_{\{A_2f\geq c_0e^{2\rho r}\}})(x)&=\frac{1}{|B(x,N+1)|}\int_{B(x,N+1)}\Chi_{\{A_2f\geq c_0e^{2\rho r}\}}(z)\,dz\\&\geq \frac{1}{|B(x,N+1)|}\int_{B(y,1)}dz=\frac{|B(y,1)|}{|B(x,N+1)|}.
     \end{align*}
     Hence, the claim \eqref{claim_A_N+1} follows from the fact that $$\frac{|B(y,1)|}{|B(x,N+1)|}\simeq e^{-2\rho N}.$$   
     We now estimate the right-hand side of \eqref{claim_A_N+1}. We observe that  
     \begin{align}
     \label{lp-eq-3}
         w&\left(\{ A_{N+1}(\Chi_{A_2f\geq c_0e^{2\rho r}})\geq c_1e^{-2\rho N}\}\right)\nonumber\\&\lesssim e^{2\rho N}\int_{H_N}A_{N+1}(\Chi_{\{A_2f\geq c_0e^{2\rho r}\}})(z)w(z)\,dz,
     \end{align}
     where $$H_N:=\{x\in S: A_{N+1}(\Chi_{A_2f\geq c_0e^{2\rho r}})(x)\geq c_1e^{-2\rho N}\}.$$ 
     In order to estimate this, we use the hypothesis on $w.$ By writing $$G=\{x\in S:A_2f(x)\geq c_0e^{2\rho r}\},$$ we see that
     \begin{align*}
     \int_{H_N}A_{N+1}(\Chi_{G})(z)w(z)\,dz &=\frac{1}{|B(z,N+1)|} 
         \mathds{1}\otimes w(\{(x,y)\in G\times H_N: d(x,y)<N+1\})\\&\lesssim e^{2(\beta-1) \rho (N+1)} w(G)^{\frac{\alpha}{p}}w(H_N)^{1-\frac{\alpha}{p}}
     \end{align*}
     which, in view of \eqref{lp-eq-3}, shows that 
     $$w(H_N)\lesssim e^{2\rho N\beta\frac{p}{\alpha}}w(G).$$
     Hence, from \ref{lp-eq-2}, we have the estimate 
     \begin{equation}
         w\left(\{A_N(A_1f\Chi_{E_{\infty}})\neq 0\}\right)\lesssim e^{2\rho N\beta\frac{p}{\alpha}} ~w(\{A_2f\geq c_0e^{2\rho N}\}).
     \end{equation}
    Certainly, this takes care of the last term in \eqref{lp-lem1eq1}. Now, to estimate the second term in \eqref{lp-lem1eq1}, we first observe that 
     if $$\sum_{k=0}^Ne^{2\rho k}A_N(\Chi_{E_k})(x)\geq e^{-1},$$ for some $x\in S$, then we necessarily have 
     $$A_{N}(\Chi_{E_k})(x)\geq \frac{e^{2\rho \gamma}-1}{e^{2\rho (k+2)}}e^{2\rho (k-N)\gamma},$$
     for some $0\leq k\leq N.$ Otherwise, we would have 
     \begin{align*}
      e^{-1}\leq   \sum_{k=0}^Ne^{2\rho k}A_N(\Chi_{E_k})(x)\leq  \frac{e^{2\rho \gamma}-1}{e^{4\rho }}\sum_{k=0}^Ne^{2\rho (k-N)\gamma}=\frac{e^{2\rho \gamma}-1}{e^{4\rho }}\frac{e^{2\rho \gamma}-e^{-2\rho N}}{e^{2\rho \gamma}-1}<e^{-1},
     \end{align*}
     which is absurd. Therefore,  
     \begin{equation}
     \label{lp-eq-5}
         w\left( \left\{ \sum_{k=0}^Ne^{2\rho k}A_N(\Chi_{E_k})\geq e^{-1}\right\}\right)\leq \sum_{k=0}^Nw(F_k),
     \end{equation}
     where $F_k$ is defined by 
     $$F_k=\{x\in S: A_{N}(\Chi_{E_k})(x)\geq \frac{e^{2\rho \gamma}-1}{e^{2\rho (k+2)}}e^{2\rho (k-N)\gamma}\}.$$
     Now, from the definition of $F_k$, it is clear that 
     $$w(F_k)\leq \frac{e^{2\rho (k+2)}}{e^{2\rho \gamma}-1}e^{-2\rho (k-N)\gamma}\int_{F_k}A_{N}(\Chi_{E_k})(x)w(x)\,dx.$$ 
     But a simple calculation shows that 
     \begin{align*}
         \int_{F_k}A_{N}(\Chi_{E_k})(x)w(x)\,dx&=\frac{1}{|B(x,N)|}\mathds{1}\otimes w(\{(x,y)\in E_k\times F_k: d(x,y)<N\})\\
         & \lesssim e^{2\rho N(\beta-1)}w(E_k)^{\frac{\alpha}{p}}w(F_k)^{1-\frac{\alpha}{p}},   
     \end{align*}
     which yields 
     \begin{align*}
         w(F_k)\lesssim \frac{e^{2\rho (k+2)}}{e^{2\rho \gamma}-1}e^{-2\rho (k-N)\gamma}e^{2\rho N(\beta-1)}w(E_k)^{\frac{\alpha}{p}}w(F_k)^{1-\frac{\alpha}{p}}.
     \end{align*}
     That is, we have 
     \begin{align*}
        w(F_k)&\lesssim \left(\frac{1}{e^{2\rho \gamma}-1}\right)^{\frac{p}{\alpha}}e^{2\rho (k+2)\frac{p}{\alpha}}e^{-2\rho (k-N)\gamma\frac{p}{\alpha}}e^{2\rho N(\beta-1)\frac{p}{\alpha}}w(E_k) \\
        & = \left(\frac{1}{e^{2\rho \gamma}-1}\right)^{\frac{p}{\alpha}}e^{2\rho (k+2-k\gamma)\frac{p}{\alpha}} e^{2\rho N(\beta-1+\gamma)\frac{p}{\alpha}}w(E_k),
     \end{align*}
     which, after choosing $\gamma=(1-\beta)/2$, transforms into 
     $$w(F_k)\lesssim e^{\rho (k-N)(1-\beta)\frac{p}{\alpha}} e^{2\rho \beta\frac{p}{\alpha}k}w(E_k).$$
     So, from \eqref{lp-eq-5} we get
     $$w\left(\left\{ \sum_{k=0}^Ne^{2\rho k}A_N(\Chi_{E_k})\geq e^{-1}\right\}\right)\lesssim \sum_{k=0}^N e^{\rho (k-N)(1-\beta)\frac{p}{\alpha}} e^{2\rho \beta\frac{p}{\alpha}k}w(E_k).$$
     Combining the inequality above with \eqref{lp-eq-3}, in view of \eqref{lp-lem1eq1}, we finally obtain
     $$w\left(\{A_N(A_1f)\geq 1\}\right) \lesssim \sum_{k=0}^{N} e^{\rho (k-N)(1-\beta)\frac{p}{\alpha}} e^{2\rho \beta\frac{p}{\alpha}k}~w\left(\{A_2f> c_0 e^{2\rho k}\}\right).$$ 
 \end{proof}
 We now employ this lemma to establish one of our main results, Theorem \ref{weightedp}. In addition to that, the subsequent theorem will demonstrate even more.
 \begin{thm}\label{thm_Lp_bdd}
 Let $1<p<\infty$, $\beta\in (0,1)$, $\beta\leq \alpha<p$. Suppose that $w$ is an admissible $A_p$ weight with parameters $(\alpha,\beta)$. Then the following assertions are true.
    \begin{enumerate}
        \item Suppose that $\beta = \alpha$. Then the maximal operator $M$ is a bounded operator from $L^p(w)$ to $L^{p,\infty} (w)$.
\item Suppose that $\beta<\alpha$. Then for any $\varepsilon>0$, we have for all $ f\in L^p(w)$
\begin{equation}\label{eq_LpA_Nf_thm}
    \sum_{N=1}^{\infty} N^{\varepsilon}\|A_{N} f\|_{L^p(w)} \lesssim \left(\sum_{N=1}^{\infty} N^{\varepsilon} e^{-2\rho pN (1-\frac{\beta}{\alpha})} \right)  \|f\|_{L^p(w)}. 
\end{equation}
Consequently, we have 
 \begin{equation}\label{Lp_of_Mf}
    \| Mf\|_{L^{p}(w)} \lesssim  \|f\|_{L^p(w)},
\end{equation}
for all $ f\in L^p(w)$, and 
\begin{equation*}\label{Lp'_of_Mf}
    \| Mf\|_{L^{p'}(\sigma_p)} \lesssim    \|f\|_{L^{p'}(\sigma_p)},
\end{equation*}
for all $ f\in L^{p'}(\sigma_p)$, where $\sigma_p =w^{-\frac{1}{p-1}}$.

    \end{enumerate}
 \end{thm}
 \begin{proof}
As already explained, the boundedness of the local part,  $M^0$, is assured by the hypothesis   $w\in A_{p, loc}(S).$  We are thus reduced to proving the followings: 
  \begin{align*}
      w \left( \left\{  M^{\infty} f \geq \lambda \right\}\right) &\lesssim{\lambda^p} \int_{S} |f(x)|^{p} w(x)\,dx,\\
      \|M^{\infty} f\|_{L^p(w)} &\lesssim \|f\|_{L^p(w)}.
  \end{align*}
for the cases $\beta=\alpha$ and $\beta <\alpha$, respectively. Without loss of generality, we assume that $f$ is non-negative.  We first prove the estimate for $\beta= \alpha$. To start with, fix $\lambda>0.$ It is not difficult to see that there exist two positive constants $c_0$, $c$ such that 
$$w \left( \left\{  M^{\infty} f \geq \lambda \right\}\right) \leq   w \left( \left\{  M^{\infty} (A_1f)\geq c\lambda \right\}\right)\leq \sum_{N=1}^{\infty}   w \left( \left\{  A_N (A_1 f) \geq c_0 \lambda \right\}\right),$$ which, by Lemma \ref{lem_A_1_leb_est} gives us
\begin{align*}
      w \left( \left\{  M^{\infty} f \geq \lambda \right\}\right)
      \lesssim  \sum_{N=0}^{\infty} \sum_{k=0}^{N} e^{\rho (k-N)(1-\beta)\frac{p}{\alpha}} e^{2\rho \beta\frac{p}{\alpha}k}~w\left(\{A_2f> \eta\lambda e^{2\rho k}\}\right).
\end{align*}
Applying Fubini's theorem, we obtain
\begin{align*}
      w \left( \left\{  M^{\infty} f \geq \lambda \right\}\right)& \lesssim \int_{S}  \sum_{N=0}^{\infty} \sum_{k=0}^{N} e^{\rho (k-N)(1-\beta)\frac{p}{\alpha}}e^{2\rho \beta\frac{p}{\alpha}k} \Chi_{\{A_2f> \eta\lambda e^{2\rho k}\}}(x)   w(x) \,dx \\
      & =  \int_{S}  \sum_{k=0}^{\infty} \left(\sum_{N=k}^{\infty} e^{\rho (k-N)(1-\beta)\frac{p}{\alpha}}\right)e^{2\rho \beta\frac{p}{\alpha}k} \Chi_{\{A_2f> \eta\lambda e^{2\rho k}\}}(x)   w(x) \,dx.
\end{align*}
Since, by the hypothesis, $\beta\in(0,1)$, we observe that the inner sum is finite and bounded by a constant, which is independent of $k$. Thus, we have obtained that  
\begin{align*}
        w \left( \left\{  M^{\infty} f \geq \lambda \right\}\right)& \lesssim \int_{S}  \sum_{k=0}^{\infty}  e^{2\rho \beta\frac{p}{\alpha}k} \Chi_{\{A_2f> \eta\lambda e^{2\rho k}\}}(x)   w(x) \, dx\\
        &\lesssim \int_{S}  \sum_{k=0}^{\infty}  e^{2\rho \beta\frac{p}{\alpha}k}    \Chi_{\{A_2f> \eta\lambda e^{2\rho k}\}}(x) |{A_2f(x)}|^{p} \lambda^{-p} { e^{-2\rho  p k}}    w(x) \, dx\\
        &\lesssim \frac{1}{\lambda^p} \int_{S}  \left( \sum_{k=0}^{\infty}  e^{2\rho  {p} k(\frac{\beta}{\alpha} -1)}   \Chi_{\{A_2f> \eta\lambda e^{2\rho k}\}}(x)\right)  {A_2f(x)}^{p} w(x) \,dx\\
         &\lesssim \frac{1}{\lambda^p} \int_{S}     |{A_2f(x)}|^{p} w(x) \,dx\\
         &\lesssim \frac{1}{\lambda^p} \int_{S}     M^{0} f(x)^p w(x) \,dx,
\end{align*}
where in the last inequality, we used the fact  ${A_2f(x)}\leq  M^{0} f(x) $. Finally, using the hypothesis $w\in A_{p, loc}(S)$, we have established  
\begin{equation*}
    \| Mf\|_{L^{p,\infty}(w)} \leq {C_p}   \|f\|_{L^p(w)},
\end{equation*}
for all $f \in L^p(w)$.

We now turn our attention to the case $\beta<\alpha$. Our first goal is to find a quantitative estimate of  $ \|A_N f\|^{p}_{L^p(w)}$. In fact, we  prove the following
 \begin{align}\label{est_AN_f_Lp}
     \|A_N f\|_{L^p(w)} & \lesssim  e^{-2\rho N (1-\frac{\beta}{\alpha})}  \|A_2f\|^p_{L^p(w)}.
\end{align}
The inequality above shows how one can crucially use the fact $\beta<\alpha$ and get an exponential decay. In particular, for any $\varepsilon \geq 0$  and using the fact $1-\beta/\alpha>0$, we can write from \eqref{est_AN_f_Lp}
\begin{align}\label{sumLp_ANf}
    \sum_{N=1}^{\infty} N^{\varepsilon} \|A_N f\|_{L^p(w)}\lesssim \left(\sum_{N=1}^{\infty} N^{\varepsilon}e^{-2\rho N (1-\frac{\beta}{\alpha})} \right) \|A_2f\|^p_{L^p(w)},
\end{align}
where the sum inside parentheses is bounded by a constant. But by the hypothesis $w\in A_{p,loc}{(S)}$, 
\begin{align*}
    \|A_2f\|_{L^p(w)}\lesssim \|M^0 f\|_{L^p(w)} \lesssim\|f\|_{L^p(w)}.
\end{align*}
Combining the inequality above with \eqref{sumLp_ANf}, we obtain
\begin{align*}
    \|M^{\infty}f\|_{L^p(w)} \leq \sum^{\infty}_{N=1} \|A_N f\|_{L^p(w)} \lesssim \left(\sum_{N=1}^{\infty} e^{-2\rho N (1-\frac{\beta}{\alpha})} \right)\|f\|_{L^p(w)}.
\end{align*}
Thus, we are done once we prove \eqref{est_AN_f_Lp}. To prove the estimate \eqref{est_AN_f_Lp}, we first write $\|A_N f\|_{L^p(w)}$ in the following way
\begin{align*}
    \|A_N f\|^{p}_{L^p(w)}&= p\int_{0}^{\infty} \lambda^{p-1} w( \{A_{N} f\geq \lambda\})\, d\lambda\\
    &\leq p\int_{0}^{\infty} \lambda^{p-1} w( \{A_{N}(A_1f) \geq c_0 \lambda\}) \,d\lambda.
\end{align*}
Now, we use the estimate \eqref{A_1_f_est} in the inequality above and then use Fubini's theorem to obtain 
\begin{align*}
     \|A_N f\|^{p}_{L^p(w)}\lesssim  \sum_{k=0}^{N} e^{\rho (k-N)(1-\beta)\frac{p}{\alpha}} e^{2\rho \beta\frac{p}{\alpha}k} \left(\int_{0}^{\infty} \lambda^{p-1}~w\left(\{A_2f> \eta\lambda e^{2\rho k}\}\right) d\lambda \right).
\end{align*}
  The change of variable $\lambda \mapsto \eta \lambda e^{2\rho k}$ in the integral above implies that
  \begin{align*}
     \|A_N f\|^{p}_{L^p(w)}& \lesssim \sum_{k=0}^{N} e^{\rho (k-N)(1-\beta)\frac{p}{\alpha}} e^{2\rho pk (\frac{\beta}{\alpha}-1)} \left(\int_{0}^{\infty} \lambda^{p-1}~w\left(\{A_2f>\lambda\}\right) d\lambda \right)\\
     & \lesssim e^{-2\rho pN (1-\frac{\beta}{\alpha})}  \|A_2f\|^p_{L^p(w)}.
\end{align*}
 Thus, we have the desired estimate \eqref{est_AN_f_Lp}.

For the last part of our theorem, we use the standard duality argument. Since, by the hypothesis, $w\in {A}_{p,loc}(S) $, $\sigma_{p}\in {A}_{p',loc}(S)$, where we recall that $\sigma_p=w^{-1/(p-1)}$. This settles the $L^{p'}(\sigma_p)$-boundedness of the local part $M^0$ of the maximal operator of $M$. 
Now, for $\psi\in L^p(w)$ 
 \begin{align*}
     \left| \int_{S} M^{\infty}f(x) \psi(x)\:dx\right|& \leq \left| \int_{S} M^{\infty}f(x) |\psi(x)| \:dx\right|
     \leq \sum^{\infty}_{N=1} A_N f(x) |\psi(x)|\:dx.
     \end{align*}
     As $A_N$ is  self-adjoint, the last inequality implies 
     \begin{align*}
     \left| \int_{S} M^{\infty}f(x) \psi(x)\:dx\right|&\leq  \sum^{\infty}_{N=1} f(x) A_N \psi(x)\:dx\\ 
     &\leq  \|f\|_{L^{p'} (\sigma_p)}\sum^{\infty}_{N=1} \| A_N \psi\|_{L^p(w)}\\ 
     &\lesssim   \|f\|_{L^{p'}(\sigma_p)} \| \psi\|_{L^p(w)},
     \end{align*}
 
 where we have used \eqref{eq_LpA_Nf_thm} in the last inequality.  We conclude our theorem using duality.
  \end{proof}
We now demonstrate that the aforementioned theorem is sharp in the sense that there exist weights in the admissible $A_p$ class with parameters $\beta=\alpha$ for which strong type $(p,p)$ does not hold. 
\begin{exmp}\label{notstrong} 
  
  Let $w$ be as in \eqref{generalweight}. Using the asymptotic estimate \eqref{sph_asymp} of $\varphi_{i\gamma} $, with $\gamma=(2p-1)\rho$, it follows that $w$ satisfies the hypothesis of Proposition \ref{easy_check}, which in turn implies via Theorem \ref{thm_Lp_bdd} that $$ \| M f \|_{L^{p,\infty} (w)} \lesssim \|f\|_{L^p(w)}.$$ Now, we show that $M$ is not a bounded operator from $L^p (w)$ to itself. This demonstrates the requirement that $M_s w \lesssim w$ is not necessary for $w$ to be in the admissible $A_p$ class with parameters $(1/2,1/2)$.

To prove $M$ is not strong type $(p,p)$, we take $f=\Chi_{B(e,1)}$. It is clear that $f\in L^{p}(w)$. We claim, for each $j\geq 1$, that   
\begin{equation}\label{Mf_lower}
    Mf(x) \gtrsim e^{-2\rho j},
\end{equation}
whenever $x\in \Omega_j$. Once we achieve this, then 
\begin{align*}
    \| Mf\|^{p}_{L^p(w)} &= \int_S \left( Mf \right)^p(y)w(y)\,dy 
    \geq \sum^{\infty}_{j=1} \int_{\Omega_j} \left( Mf \right)^p(y)w(y)\,dy
    \gtrsim \sum^{\infty}_{j=1} \int_{\Omega_j}e^{-2p\rho j } e^{2\rho (p-1)j} \,dy,
\end{align*}
where we have used \eqref{Mf_lower}. As $|\Omega_j|\simeq e^{2\rho j}$, it follows that
$$\| Mf\|^{p}_{L^p(w)}\gtrsim \sum^{\infty}_{j=1} e^{-2\rho j}|\Omega_j|\gtrsim \sum^{\infty}_{j=1} 1=\infty.$$
Therefore, it remains to show the estimate \eqref{Mf_lower} of $Mf$. From the definition 
\begin{align*}
   Mf(x) &= \sup_{r>0} \frac{1}{\left| B(x,r)\right|} \int_{B(x,r)} \Chi_{B(e,1)}(y) dy\\
   & = \sup_{r>0} \frac{ \left| B(x,r)\cap  B(e,1)\right| }{\left| B(x,r)\right|}
\end{align*}
Now, for $x\in \Omega_j$, we observe that the set ${ B(x,r)\cap  B(e,1)}$ is empty, whenever $r<j-1$. On the other hand, $$ { B(x,r)\cap  B(e,1)} = B(e,1), $$  for all  $r>j+1$. Therefore, the supremum above is essentially over $j-1\leq r\leq j+1$. But in this range of $r$, we clearly have 
\begin{align*}
    Mf(x) \simeq  e^{-2\rho j},
\end{align*}
which concludes our claim \eqref{Mf_lower}.
\end{exmp}

 Next, we provide an illustration to support the assertion that the classical $A_p$ condition is not necessary for the weighted $L^p$-boundedness of $M$.
\begin{exmp}\label{Apnot}
We consider the weight $$w(x) = e^{2\rho \gamma \, d(e, x)},\hspace{0.3cm}x\in S,$$ for some $ \gamma \in (-1,-1/2)$  fixed. Then from  \eqref{blesa}, for any $1<p<\infty$, it follows that $w $ belongs to the admissible $A_p$ class with parameters as in Example \eqref{blesa}, and consequently,  by Theorem \ref{thm_Lp_bdd} we also get the weighted $L^p$-boundedness of $M.$
We now show that $w$ does not satisfy the classical $A_p$ condition. For $x\in \Omega_j$, we observe that  
\begin{equation}\label{int_w>}
\begin{aligned}
    \frac{1}{|B(x,j)|} \int_{B(x,j)} w(y)\,dy& \geq \frac{1}{|B(x,j)|} \int_{B(x,j)\cap B(e,1)} w(y)\,dy
     \simeq \frac{\left| {B(x,j)\cap B(e,1)}\right|}{|B(x,j)|}
     \gtrsim \frac{1}{e^{2\rho j}}.
\end{aligned}
\end{equation}
Also, 
\begin{align*}
    \left(  \frac{1}{|B(x,j)|} \int_{B(x,j)} w(y)^{-\frac{1}{p-1}} \,dy\right)^{p-1} & \geq  \left(  \frac{1}{|B(x,j)|} \int_{B(x,j)\cap \Omega_{2j}} w(y)^{-\frac{1}{p-1}}\,dy \right)^{p-1}.
\end{align*}
Plugging the estimate of $ \left| {B(x,j)\cap \Omega_{2j}} \right|$ from \eqref{annular} in the equation above, we obtain \begin{equation}\label{int_w^{p-1}}
\begin{aligned}
    \left(  \frac{1}{|B(x,j)|} \int_{B(x,j)} w(y)^{-\frac{1}{p-1}} \,dy\right)^{p-1}
    \gtrsim  \left( \frac{e^{-4\rho \gamma j \frac{1}{p-1}}}{e^{2\rho j}} e^{2\rho j} \right)^{p-1}
    = e^{-4\rho \gamma j }
\end{aligned}
\end{equation}
Combining \eqref{int_w>} and \eqref{int_w^{p-1}}, we get
\begin{align*}
     \left(\frac{1}{|B(x,j)|} \int_{B(x,j)} w(y)~dy\right)  \left(  \frac{1}{|B(x,j)|} \int_{B(x,j)} w(y)^{-\frac{1}{p-1}} \,dy\right)^{p-1} \gtrsim \frac{1}{e^{2\rho j}}  e^{-4\rho \gamma j } = e^{-2\rho j (2\gamma +1)}.
\end{align*}
Therefore, letting $j\rightarrow \infty$, we see that $w$ does not satisfy the classical $A_p$ condition.
\end{exmp}
Finally, we finish this section with an example that  demonstrates that the necessary condition \eqref{neceap} is not sufficient for the weighted $L^p$-boundedness of $M.$
\begin{exmp}\label{growthnec}
 Let $w(x) = e^{-2\rho d(e,x)}$. Repeating the argument verbatim as in Proposition \ref{easy_check}, it can be proved that  $w$  satisfies 
  \begin{align*}
     \mathds{1}\otimes w(\{(x,y)\in E\times F: d(x,y)<N\}) \lesssim e^{2\rho N} w(E)^{\frac{1}{p}}w(F)^{1-\frac{1}{p}},
 \end{align*}
 for any measurable subsets $E, F\subseteq S.$ But as we will see later in the proof of Theorem \ref{thm_FS_not_1} in the next section that $M$ is not even weighted weak type $(p,p)$ for any $p\in (1,\infty)$.
\end{exmp}

\section{Fefferman-Stein type estimates}\label{fsineq}
In this section, we prove the variant of Fefferman-Stein inequality alluded to in the introduction. We begin with a variation of Lemma \ref{A_1_f_est}.
 \begin{lem}\label{msan}
  Let $w$ be a weight and $s>1.$ Then for each $\lambda>0$, and $N\in \mathbb{N}$, there exists a constant $\eta>0$ such that   
     \begin{equation}\label{weak-A_1_f_est}
         w\left(\{A_N(A_1f)\geq \lambda\}\right) \lesssim  \sum_{k=0}^{N} e^{-\rho \frac{N}{s'}} e^{2\rho k(1+\frac{1}{2s'})}~M_sw\left(\{A_2f> \eta\lambda e^{2\rho k}\}\right).
     \end{equation}
 \end{lem}
 \begin{proof}
 The proof of the Lemma \ref{A_1_f_est} may be reproduced exactly in this situation; just a little adjustment is required to include $M_s$ on the right-hand side. As in the proof of Lemma \ref{A_1_f_est}, we have 
 \begin{equation}
     \label{weak-lp-lem1eq1}
         w\left(\left\{A_N(A_1f)\geq 1 \right\}\right) 
         \leq w\left( \left\{ \sum_{k=0}^Ne^{2\rho k}A_N(\Chi_{E_k})\geq e^{-1}\right \}\right)+ w\left(\{A_N(A_1f\Chi_{E_{\infty}})\geq 1-2e^{-1}\}\right),
     \end{equation}
 where the last term can be estimated as 
     \begin{align*}
         w&\left(\{A_N(A_1f\Chi_{E_{\infty}})\geq 1-2e^{-1}\}\right)\\&\lesssim e^{2\rho N}\int_{S}A_{N+1}(\Chi_{\{A_2f\geq c_0e^{2\rho N}\}})(z)w(z)\,dz.
     \end{align*}
     But the above integral, in view of the self-adjointness of $A_{N+1}$, is dominated by 
     $$\int_{S}\Chi_{\{A_2f\geq c_0e^{2\rho r}\}}(z)Mw(z)\,dz.$$
     Using $Mw\leq M_sw$,  we obtain 
     \begin{equation}
     \label{weak-l1-eq1}
         w\left(\{A_N(A_1f\Chi_{E_{\infty}})\geq 1-2e^{-1}\}\right)\leq e^{2\rho N} M_sw\left( \{A_2f\geq c_0e^{2\rho N}\}\right).
     \end{equation}
     Now, to estimate the other term in \eqref{weak-lp-lem1eq1}, we proceed as exactly as in the Lemma \ref{weak-A_1_f_est}, to get 
     \begin{equation}
     \label{weak-l1-eq3}
         w\left(\left\{ \sum_{k=0}^Ne^{2\rho k}A_N(\Chi_{E_k})\geq e^{-1}\right\}\right)\leq \sum_{k=0}^N w(F_k),
     \end{equation}
     where 
     $$w(F_k)\leq \frac{e^{2\rho (k+2)}}{e^{2\rho \gamma}-1}e^{-2\rho (k-N)\gamma}\int_{F_k}A_{N}(\Chi_{E_k})(x)w(x)\,dx,$$
 for $\gamma \in (0,1)$ to be specified at a later stage. At this point, we use Lemma \ref{weak-l1-lem1} to estimate the above integral, obtaining 
 \begin{align*}
     w(F_k)&\lesssim\frac{e^{2\rho (k+2)}}{e^{2\rho \gamma}-1}e^{-2\rho (k-N)\gamma}e^{-2\rho \frac{N}{s'+1}}w(F_k)^{\frac{1}{s'+1}}M_sw(E_k)^{\frac{s'}{s'+1}}
 \end{align*}
 which then implies 
 \begin{align*}
     w(F_k)&\lesssim\left(\frac{e^{2\rho (k+2)}}{e^{2\rho \gamma}-1}\right)^{\frac{s'+1}{s'}}e^{-2\rho (k-N)\gamma\frac{s'+1}{s'}}e^{-2\rho \frac{N}{s'}}M_sw(E_k).
 \end{align*}
We now simplify the right-hand side of the inequality above by making an appropriate choice of $\gamma.$  More precisely, choosing $\gamma =\frac{1}{2(s'+1)}$, after a simple calculation we arrive at 
\begin{align*}
    w(F_k)\lesssim e^{-\rho \frac{N}{s'}} e^{2\rho k(1+\frac{1}{2s'})} M_sw(E_k),
\end{align*}
which, when coupled with \eqref{weak-lp-lem1eq1}, \eqref{weak-l1-eq1}, and \eqref{weak-l1-eq3}, proves the result. 
 \end{proof}
 \textbf{\textit{Proof of Theorem \ref{fsna}.}}
 It follows from the proof of the  Fefferman-Stein on Euclidean spaces that the inequality \eqref{fsna} for $M^0$ holds. Thus, it remains to prove the \eqref{fsna} for $M^{\infty}$.
 We have already observed in the proof of Theorem \ref{weightedp} that there exist two positive constants $c_0$, $c$ such that
    \begin{align*}
        w \left( \left\{  M^{\infty} f \geq \lambda \right\}\right)
        \leq   w \left( \left\{  M^{\infty} (A_1f) \geq 
 c\lambda \right\}\right)
      \leq \sum_{N=1}^{\infty}   w \left( \left\{  A_N (A_1f) \geq c_0\lambda \right\}\right),
       \end{align*}
       where the last sum, in view of the Lemma \ref{msan}, is dominated by a constant multiple of 
       $$\sum_{N=0}^{\infty} \sum_{k=0}^{N} e^{-\rho \frac{N}{s'}} e^{2\rho k(1+\frac{1}{2s'})}~M_sw\left(\{A_2f> \tau\lambda e^{2\rho k}\}\right),$$
     for some $\tau>0$. After changing the order of the summations, the above takes the form 
\begin{align*}
   &  \sum_{k=0}^{\infty} \left(\sum_{N=k}^{\infty} e^{\rho (k-N)\frac{1}{s'}}\right)e^{2\rho k} \int_{S}\Chi_{\{A_2f> \tau\lambda e^{2\rho k}\}}(x)   M_sw(x)\,dx\\
   &= \int_{S}  \sum_{k=0}^{\infty}e^{2\rho k}\Chi_{\{A_2f> \tau\lambda e^{2\rho k}\}}(x)   M_sw(x)\,dx\\
   &\leq \frac{1}{\tau\lambda} \int_{S}A_2f(x)M_sw(x)\,dx.
\end{align*}
Using the self-adjointness of $A_2$, we obtain 
\begin{align}\label{weaka2}
    w \left( \left\{  M^{\infty} f \geq \lambda \right\}\right)\lesssim \frac{1}{\lambda} \int_{S}|f(x)|A_2(M_sw)(x)\,dx.
\end{align}
For the choice of $w\equiv 1$, $A_2(M_sw)\equiv 1$. Thus, we recover the result due to Anker et al. \cite{ADY} that $M:L^1(S)\rightarrow L^{1,\infty}(S)$ is bounded. For the general case, we bound  
\begin{align}\label{a2ms1}
A_2(M_sw)(x)=\frac{1}{|B(x,2)|}\int_{B(x,2)}\left(M(w^s\Chi_{B(x,4)})(y)^{\frac{1}{s}}+M(w^s\Chi_{B(x,4)^c})(y)^{\frac{1}{s}}\right)\,dy
\end{align}
where the last term is clearly dominated by 
\begin{equation}\label{domin}
    M(w^s\Chi_{B(x,4)^c})(x))^{\frac{1}{s}}\lesssim M_sw(x).
\end{equation} We are now only left with the first term. Since we have already proved the unweighted weak type $(1,1)$ estimate, we can apply Proposition \ref{kolmogorov} to get
\begin{align}\label{a2ms2}
    \frac{1}{|B(x,2)|}\int_{B(x,2)}(M(w^s\Chi_{B(x,4)})(y)^{\frac{1}{s}}dy\leq \frac{\|M\|_{L^1\rightarrow L^{1,\infty}}}{1-(1/s)}\frac{|B(x,4)|^{1-1/s}}{|B(x,2)|}\|w^s\Chi_{B(x,4)}\|_{L^1(S)}^{\frac1s}.
\end{align}
We observe that 
\begin{equation*}
    \|w^s\Chi_{B(x,4)}\|_{L^1(S)}^{\frac1s}=(A_4(w^s)(x))^{\frac1s}\leq M_sw(x).
\end{equation*}
We use this observation in \eqref{a2ms2} to obtain 
$$\frac{1}{|B(x,2)|}\int_{B(x,2)}(M(w^s\Chi_{B(x,4)})(y))^{\frac{1}{s}}\,dy\leq \frac{\|M\|_{L^1\rightarrow L^{1,\infty}}}{1-(1/s)}\frac{|B(x,4)|^{1-1/s}}{|B(x,2)|}M_sw(x).$$
Applying the inequality above together with \eqref{domin} in \eqref{a2ms1} yields
 $$A_2(M_sw)(x)\leq C_s M_sw(x),$$ where 
 \begin{equation}
     \label{csform}
     C_s\geq\frac{s}{1-s} \frac{|B(e,4)|^{1-1/s}}{|B(e,2)|}\|M\|_{L^1\rightarrow L^{1,\infty}}.
 \end{equation}
Finally, plugging in this estimate in \eqref{weaka2}, we get
\begin{align*}
    w \left( \left\{  M^{\infty} f \geq \lambda \right\}\right)\lesssim \frac{C_s}{\lambda} \int_{S}|f(x)|M_sw(x)\,dx
\end{align*}
where from \eqref{csform}, it readily follows that $C_s\rightarrow\infty$ as $s\rightarrow 1.$ 
This completes the proof.
\qed

\vspace{0.3cm}
We apply Marcinkiewicz's interpolation theorem to obtain the following corollary.
 \begin{cor}
     \label{p-fsineq}
     Let $w$ be a weight on $S$, and $s>1.$ Then for any $1<p<\infty$, we have 
     $$\int_S Mf(x)^pw(x)\,dx\leq C_{s,p} \int_S |f(x)|^pM_sw(x)\,dx,$$
  where $C_{s,p}\to\infty$ as $s\to 1$.   
 \end{cor}
As an immediate consequence of the corollary above, we prove the vector-valued maximal inequality for $M$, i.e., Theorem \ref{vectorvalued}.

\textbf{\textit{Proof of Theorem \ref{vectorvalued}:}}	
Let us first deal with the case $p=r.$ Using the usual maximal inequality, we note that 
\begin{align*}
\int_{S}\sum_{n=1}^{\infty}Mf_n(x)^r=\sum_{n=1}^{\infty}\|Mf_n\|_{L^r(S)}^r\,dx\lesssim_r\sum_{n=1}^{\infty}\|f_n\|_{L^r(S)}^r=\int_{S}\sum_{n=1}^{\infty}|f_n(x)|^r\,dx
\end{align*}
from which, the result follows immediately. On the other hand, when $1<r<p<\infty,$ we first observe that
\begin{equation}
    \label{vecfs1}
    \left\|\left(\sum_{n=1}^{\infty}Mf_n(\cdot)^r\right)^\frac{1}{r}\right\|_{L^p(S)}^r=\left\|\left(\sum_{n=1}^{\infty}Mf_n(\cdot)^r\right)\right\|_{L^{\frac{p}{r}}(S)}.
\end{equation}
But by duality,  the last term equals to  
$$\sup_{\|\psi\|_{L^{(p/r)'(S)}}\leq 1}\left|\int_S\left(\sum_{n=1}^{\infty}Mf_n(x)^r\right)\psi(x)\:dx \right|.$$
Now, to estimate the above, we observe that 
\begin{align*}  \left|\int_S\left(\sum_{n=1}^{\infty}Mf_n(x)^r\right)\psi(x)\:dx \right|&\leq \sum_{n=1}^{\infty}\int_SMf_n(x)^r|\psi(x)|\:dx \\
    &\lesssim_{r,s} \sum_{n=1}^{\infty}\int_S|f_n(x)|^rM_s\psi(x)\:dx,
\end{align*}
where in the last inequality we have used Corollary \ref{p-fsineq}. We now apply H\"older inequality to the last term in the inequality above to get 
\begin{align*}
    \sum_{n=1}^{\infty}\int_S|f_n(x)|^rM_s\psi(x)dx\leq \|M_s\psi\|_{L^{(p/r)'}(S)}\left(\int_S\left(\sum_{n=1}^{\infty}|f_n(x)|^r\right)^{\frac{p}{r}}dx\right)^{\frac{r}{p}}. 
\end{align*}
But by the definition of $M_s$, we notice that 
$$\|M_s\psi\|_{L^{(p/r)'}(S)}= \|M(|\psi|^s)^{1/s}\|_{L^{(p/r)'}(S)}=\left(\int_S(M(|\psi|^s)(x))^{(r/p)'/s}\right)^{1/(r/p)'}.$$
We take $s<(r/p)'$ so that by the boundedness of $M$, 
$$\left(\int_S(M(\psi^s)(x))^{(p/r)'/s}\right)^{1/(p/r)'}\lesssim_{p,r} \|\psi\|_{L^{(p/r)'}(S)}.$$
Therefore, we obtain 
$$\left|\int_S\left(\sum_{n=1}^{\infty}Mf_n(x)^r\right)\psi(x)\:dx \right|\lesssim_{p,r}\|\psi\|_{L^{(p/r)'}(S)}  \left\|\left(\sum_{n=1}^{\infty}|f_n(\cdot)|^r\right)\right\|_{L^{p/r}(S)},$$
which yields
$$\sup_{\|\psi\|_{L^{(p/r)'}(S)}\leq 1}\left|\int_S\left(\sum_{n=1}^{\infty}|Mf_n(x)|^r\right)\psi(x)\:dx \right|\lesssim_{p,r} \left\|\left(\sum_{n=1}^{\infty}|f_n(\cdot)|^r\right)\right\|_{L^{p/r}(S)} \|\psi\|_{L^{(p/r)'}(S)}.$$
Hence, the result follows from \eqref{vecfs1}.\qed

We now proceed to the proof of our last result. This, in particular, will also show that the Theorem \ref{fsna} is sharp in the sense that \eqref{fsineq} cannot be proved with $s=1$, validating that the exact analogue of the classical Fefferman-Stein inequality in the setting of Harmonic $NA$ groups is not possible.

 \textbf{\textit{Proof of Theorem \ref{thm_FS_not_1}}:}
 We prove for the  $p=1$ case only as for the $p\neq 1$ case, the same proof can be reproduced with obvious modification. We shall show that the weight defined by $$w(x)=e^{-2\rho d(e,x)}$$ serves our purpose. Following a similar calculation as in \eqref{claimwgama}, it is easy to see that the weight $w$ satisfies the following estimate
\begin{equation}\label{Mw<w}
M^{(k)} w \lesssim_k w.
\end{equation}
Let $f_j(x) =\Chi_{\Omega_j}(x)$. Then we see that
\begin{equation}\label{f_j_L1_bound}
  \int_{S} f_j(x) w(x)\,dx =\int_{\Omega_j} w(x) dx \simeq 1. 
\end{equation}
We now  use \eqref{Mw<w} and \eqref{f_j_L1_bound} to conclude
\begin{equation}\label{fjmbdd}
    \int_{S} f_j(x) M^{(k)} w(x)\, dx \lesssim_k 1.
\end{equation}
On the other hand, we claim
\begin{equation}\label{Omgega_j_sub}
    \bigcup^{j}_{i=0} \Omega_i \subset \left\{ Mf_{j}>1\right\}.
\end{equation}
Assuming the containment above, \eqref{eqn_not_FS} follows. Indeed, from \eqref{fjmbdd}, we get
\begin{align*}
  j \int_{S} f_i(x) M^{(k)} w(x)\, dx\lesssim_k j \lesssim_k \sum^{j}_{i=0} w(\Omega_i) \simeq_k w\left( \bigcup^{j}_{i=0} \Omega_i \right) \lesssim_k  w\left( \left\{ Mf_j(x) >1 \right\} \right).
\end{align*}
Therefore, we are left with the task of proving \eqref{Omgega_j_sub}. But before we prove \eqref{Omgega_j_sub}, we first observe that it is enough to show 
\begin{align}\label{Omega_j_C} 
    \bigcup^{j}_{i=0} \Omega_i \subset \left\{ Mf_{j}>C\right\}
\end{align}
for some fixed constant $C>0$. Because then we can take $\frac{1}{C}f_j$ instead of $f_j$, and the proof will follow similarly. To show  \eqref{Omega_j_C}, take $ x \in \Omega_i$, for some $0\leq i\leq j$. Then 
\begin{align*}
    Mf_j(x) &= \sup_{r>0} \frac{1}{|B(x,r)|} \int_{B(x,r)} f_j(x)\\
    & =  \sup_{r>0} \frac{|B(x,r)\cap \Omega_j|}{|B(x,r)|}\\
    & \geq \frac{|B(x,j-i+1)\cap \Omega_j|}{|B(x,j-i+1)|}
\end{align*}
Now, we apply \eqref{annular} to find the estimate of $|B(x,j-i+1)\cap \Omega_j|$, which in turn gives us
\begin{align*}
    Mf_j(x)  \geq C \frac{e^{2\rho(j-i+1)}}{e^{2\rho(j-i+1)}}=C,
\end{align*}
where $C$ is a constant independent of $j$. This concludes \eqref{Omega_j_C} and hence Theorem \ref{thm_FS_not_1}. \qed 
 
    \vspace*{3mm} Lastly, we pack up by making a few observations and outlining a few open questions that, in our opinion, require further investigation.
\begin{enumerate}
    \item  Alongside the centered maximal operator, there is a non-centered variant defined by 
    $$\widetilde{M}f(x)=\sup_{B\ni x }\frac{1}{|B|}\int_{B}|f(y)|dy,\quad x\in S,$$
    where the supremum is taken over all balls containing $x$. In contrast to the doubling situation, when the manifold experiences exponential volume growth, this is not comparable to the centered one, necessitating a different investigation. In the context of Riemannian symmetric spaces of non-compact type, Ionescu \cite{Ior1, Iorh} studied the unweighted boundedness of $\widetilde{M}$ which turned out to be of a different flavor than that of the centered maximal operator. In the context of Harmonic $NA$ groups, weighted boundedness of $\widetilde{M}$ has been explored in \cite{GRS}.
    \item As we have mentioned in the introduction (see Remark \ref{introrem}, (3)), there are certain limitations that are stopping us from extending the results of this article to the setting of higher rank Riemannian symmetric spaces of non-compact type. It would be intriguing to explore the possibility of circumventing certain techniques used in this paper, in order to establish some weighted boundedness for both the centered and non-centered Hardy-Littlewood maximal operators.
\end{enumerate}

\section*{Acknowledgments}
PG is supported by supported by the University of Paderborn.  TR is supported by the FWO Odysseus 1 grant G.0H94.18N: Analysis and Partial Differential Equations, the Methusalem program of the Ghent University Special Research Fund (BOF), Project title: BOFMET2021000601. JS is supported by the INSPIRE faculty fellowship (Ref. no. DST/INSPIRE/04/2022/002544) from the Department of Science and Technology, Government of India.
\bibliography{Reference.bib}
\bibliographystyle{alphaurl}
\end{document}